\documentclass[final]{siamltex}

\usepackage{amsmath,amsfonts,amssymb}
\usepackage[notcite,notref]{showkeys}
\usepackage[linesnumbered,ruled,vlined]{algorithm2e}
\usepackage{graphicx}
\usepackage{subcaption}
\usepackage{float}
\usepackage{soul}
\usepackage{mathrsfs}
\usepackage{xcolor}
\usepackage[ruled,vlined]{algorithm2e}
\usepackage{dsfont}
\usepackage{enumitem}
\usepackage{multirow}
\usepackage{multicol}

\usepackage{thmtools,thm-restate}
\declaretheorem[style=remark]{remark}

\makeatletter
\def\addcontentsline#1#2#3{
  \addtocontents{#1}{\protect\contentsline{#2}{#3}{\thepage}}}
\long\def\addtocontents#1#2{%
  \protected@write\@auxout
  {\let\label\@gobble \let\index\@gobble \let\glossary\@gobble}%
  {\string\@writefile{#1}{#2}}}

\newcommand\tableofcontents{%
  \section*{
    \@mkboth{\noindent%
       \MakeUppercase\contentsname}{\MakeUppercase\contentsname}}%
  \@starttoc{toc}%
  }
\newcommand*\l@section[2]{%
  \ifnum \c@tocdepth >\z@
  \addpenalty\@secpenalty
  \addvspace{1.0em \@plus\p@}%
  \setlength\@tempdima{1.5em}%
  \begingroup
    \parindent \z@ \rightskip \@pnumwidth
    \parfillskip -\@pnumwidth
    \leavevmode \bfseries
    \advance\leftskip\@tempdima
    \hskip -\leftskip
    #1\nobreak\hfil \nobreak\hb@xt@\@pnumwidth{\hss #2}\par
  \endgroup
  \fi}
\newcommand*\l@subsection{\@dottedtocline{2}{1.5em}{2.3em}}
\newcommand*\l@subsubsection{\@dottedtocline{3}{3.8em}{3.2em}}
\newcommand*\l@paragraph{\@dottedtocline{4}{7.0em}{4.1em}}
\newcommand*\l@subparagraph{\@dottedtocline{5}{10em}{5em}}
\makeatother

\newcommand{\email}[1]{\protect\href{mailto:#1}{#1}}
\usepackage[plainpages=false,colorlinks=true,citecolor=blue,
  filecolor=blue,linkcolor=blue,urlcolor=blue]{hyperref}

\def\addressices{Oden Institute for Computational Engineering \& Sciences,
  The University of Texas at Austin, Austin, TX, USA}
\def\addressomar{Oden Institute for Computational Engineering \& Sciences,
  Department of Mechanical Engineering, and Department of Geological
  Sciences, The University of Texas at Austin, Austin, TX, USA}

\begin{document}

\author{ Thomas O'Leary-Roseberry\footnotemark[1]  \and Nick Alger\footnotemark[2] \and Omar Ghattas\footnotemark[3]}
\renewcommand{\thefootnote}{\fnsymbol{footnote}}
\footnotetext[1]{\addressices. Email: \email{tom@ices.utexas.edu}.}
\footnotetext[2]{\addressices. Email: \email{nalger@ices.utexas.edu}.}
\footnotetext[3]{\addressomar. Email: \email{omar@ices.utexas.edu}.}
\renewcommand{\thefootnote}{\arabic{footnote}}

\title{Inexact Newton Methods for\\
 Stochastic Nonconvex Optimization with\\
Applications to Neural Network Training}


\maketitle

\begin{abstract}

We study stochastic inexact Newton methods and consider their
application in nonconvex settings. Building on the work of
[R. Bollapragada, R. H. Byrd, and J. Nocedal, IMA Journal of Numerical
  Analysis, 39 (2018), pp. 545--578] we derive bounds for convergence rates in
expected value for stochastic low rank Newton methods, and
stochastic inexact Newton Krylov methods. These bounds
quantify the errors incurred in subsampling the Hessian and gradient,
as well as in approximating the Newton linear solve, and in choosing
regularization and step length parameters. We deploy these methods in
training convolutional autoencoders for the MNIST and CIFAR10
data sets. Numerical results demonstrate that, relative to first order
methods, these stochastic inexact Newton methods often converge
faster, are more cost-effective, and generalize better.

\end{abstract}

\begin{keywords}
Machine learning, stochastic optimization, nonconvex optimization, Newton methods, Krylov methods, low rank methods, inexact Newton methods, deep learning, neural networks.
\end{keywords}

\begin{AMS}
49M15,  
49M37,  
65C60,  
65K10, 
90C26  
\end{AMS}

\pagestyle{myheadings}
\thispagestyle{plain}

\section{Introduction}

We consider the stochastic optimization problem
\begin{equation} \label{expected_risk}
  \min\limits_{w \in \mathbb{R}^d} F(w) = \int \ell(w; x,y)\,d\nu(x,y), 
\end{equation}
where $\ell$ is a smooth (loss) function, $w \in\mathbb{R}^d$ is the vector of optimization variables, the \textit{data pairs} $(x,y)$ is distributed with joint probability distribution $\nu(x,y)$, and $F:\mathbb{R}^d \rightarrow \mathbb{R}$ is referred to as the \textit{expected risk}. This problem arises in machine learning, where the goal is to reconstruct a mapping $x \mapsto y$ with a deep neural network or other model parametrized by $w$. See for example \cite{GoodfellowBengioCourville2016}. In practice, complete information about $\nu$ is not available. Rather, one has access to samples $x_i,y_i \sim \nu$, which leads to the Monte Carlo approximation of \eqref{expected_risk} as
\begin{equation} \label{empirical_risk_minimization}
  \min\limits_{w \in \mathbb{R}^d} F_{X}(w) = \frac{1}{N_X}\sum\limits_{i =1}^{N_X}F_i(w), 
\end{equation}
where $F_i(w) =\ell(w;x_i,y_i) $, and $X = \{(x_i,y_i)|x,y \sim \nu\}_{i=1}^{N_X}$ . The function $F_{X}:\mathbb{R}^d \rightarrow \mathbb{R}$ is referred to as the \textit{empirical risk}. Because of sampling error, a solution to \eqref{empirical_risk_minimization} may be a poor approximation of a solution to \eqref{expected_risk}. Iterative methods for solving \eqref{empirical_risk_minimization} are therefore judged not only by how efficently they solve \eqref{empirical_risk_minimization}, but also how well the approximate solutions they find generalize to unseen data. Optimization problem \eqref{empirical_risk_minimization} is typically solved via an iteration of the form
\begin{equation} \label{gradient_iteration}
  w_{k+1} = w_k + \alpha_k p_k,
\end{equation}
where $p_k$ is typically a gradient based search direction and $\alpha_k$ is the step length (or learning rate as it is known in machine learning). If $p_k = - \nabla F_X(w_k)$, then the iteration is gradient descent. If $p_k = - \nabla^2 F_X(w)^{-1} \nabla F_X(w)$, then the iteration is Newton's method. For many machine learning representations, the cost of evaluating $F_X$ is $O(dN_X)$. Computing the gradient using adjoint methods requires $O(dN_X)$ work and $O(d)$ storage, while explicitly forming the Hessian matrix requires $O(d^2N_X)$ work and $O(d^2)$ storage.

Several features make optimization problem \eqref{empirical_risk_minimization} difficult to solve:

\begin{enumerate}
  \item Large parameter dimension, $d$ 
  \item Large data dimension, $N_X = |X|$
  \item Nonconvexity of $F_X$
  \item Ill conditioning of $F_X$
  
\end{enumerate}

Features 1 and 2 make gradient based optimization methods expensive. To ease the computational burden, and motivated by the redundancy in the data for large $N_X$, at a given iteration $k$, one typically subsamples data $X_k$ from $X$, then substitutes $X_k$ for $X$ in \eqref{empirical_risk_minimization}. Such methods are known as stochastic gradient or stochastic Newton methods. Because of feature 3, finding global minimizers is computationally intractible (NP-hard), and instead one has to settle for local minima \cite{Bertsekas1997,MurtyKabadi1987}. In deep learning, it is conjectured that local minima are almost as good as global minima \cite{ChoromanskaHenaffMathieu2015}. However, finding local minima is still difficult because the energy landscape is riddled with saddle points \cite{DauphinPescanuGulcehre2014}. Saddle points typically correspond to suboptimal solutions in nonconvex optimization problems such as matrix factorization and phase retrieval \cite{JainJinKakadeEtAl2015,SunQuWright2016}. 

How best to deal with saddle points is an open area of research. Some
work has been done to classify when nonconvex problems are tractable
\cite{SunQuWright2015a}. In numerical optimization, modified Newton
methods that enforce positive definiteness of the Hessian (for example by maintaining positive eigenvalues in a spectral decomposition) are employed
to ensure descent for nonconvex problems
\cite{GillMurrayWright1981,NocedalWright2006}. In machine learning,
the use of the spectral decomposition of the Hessian with absolute
values of eigenvalues was introduced by Dauphin et al.\ under the name
of saddle free Newton (SFN) \cite{DauphinPescanuGulcehre2014}. This
method detects indefiniteness and facilitates quicker escape from
saddle points by following negative curvature directions. Paternain et al. prove that a variant of the SFN algorithm converges to local minima with probability $1-p$ in $O(\log(1/p)) + O(\log(1/\epsilon))$ iterations \cite{PaternainMokhtariRibeiro2019}.  Reddi et
al.\ argue for using first order methods while the gradient is large
and switching to second order methods when near stationary points
\cite{ReddiZahirSraEtAl2017}. Jin et al.\ argue for adding noise
uniformly sampled from a ball with radius large enough to dominate
saddle regions, where optimizers get stuck \cite{JinChiGeEtAl2017}.

In addition to non-convexity, feature 4 makes the optimization problem hard to solve. Gradient based methods converge very slowly for ill-conditioned
problems
\cite{Alger2019,LiChenCarlsonetal2016,Luenberger1984,SaarinenBramleyCybenko1993}.  In
contrast, under mild assumptions the convergence of Newton's method is
independent of the conditioning of the problem
\cite{NocedalWright2006}. Conventional implementations of Newton's
method are computationally impractical since explicitly forming and
factorizing a Hessian requires $O(d^2N_X + d^3)$ operations. To
address the prohibitive nature of explicit Hessian-forming Newton
methods for solving \eqref{expected_risk}, Hessian-free methods that approximate the Newton solve without explicitly forming the Hessian have been used in machine learning. Exploiting the fact that Hessian-vector products can be computed matrix-free via adjoint methods with the same complexity as gradient evaluations, i.e. $O(dN_X)$, these methods can reduce the complexity required for the Hessian approximation to $O(kdN_X)$, where $k \ll d$ is the number of matrix vector products used in the approximation. Martens et al.\ argue for the use of the Gauss-Newton Hessian as well as Kronecker factorizations of approximate curvature in approximating the Newton solve \cite{Martens2010,MartensGrosse2015,MartensSutskever2012}. Roosta et al.\ explore stochastic inexact Newton methods, and derive probabilistic bounds for spectral convergence of the subsampled Hessian to the true Hessian \cite{RoostaMahoney2016a,RoostaMahoney2016b,XuRoostaMahoney2017b,XuRoostaMahoney2017a}. In addition, they experimentally demonstrate that while Gauss Newton methods may help with the conditioning of the optimization problem, they are prone to getting stuck at saddle points. Bollapraga et al.\ analyze inexact Newton conjugate gradient (CG) methods in the semi-stochastic setting (where the Hessian is subsampled, but not the gradient) \cite{BollapragadaByrdNocedal2018}. In these
methods, CG iterations for solving the Newton system are terminated
early, thereby reducing the complexity to $O(kdN_X)$ operations, where
$k$ is the number of CG iterations, which depends on the clustering of
the eigenvalues of the Hessian.

Here we extend the analysis of
\cite{BollapragadaByrdNocedal2018} to the cases of  nonconvex
problems arising in the training of neural networks,  subsampling of
gradients (in addition to Hessians), and to other Krylov
solvers such as MINRES and GMRES. These methods can be adapted to the nonconvex case by terminating early when negative curvature directions are detected.
Moreover, as another form of inexactness,  we propose
stochastic low rank Newton algorithms that use a randomized
eigensolver to truncate the spectrum of the Hessian, retaining
eigenvalues that are largest in magnitude. Taking their absolute
values then  defines the low rank saddle free Newton algorithm. The complexity
is $O(rdN_x + r^2d)$, where $r$ is the effective rank of the
Hessian.
For the inexact Newton Krylov methods as well as the low rank
Newton, we derive convergence rates that quantify the effects on
convergence of subsampling the gradient and Hessian, approximating the
Newton linear solve, and hyperparameters such as regularization parameter and step
length $\alpha_k$.

We present numerical experiments on training convolutional
autoencoders for the MNIST and CIFAR10 data sets comparing
stochastic low rank saddle free Newton, stochastic inexact Newton CG, stochastic inexact Newton GMRES, and stochastic inexact Newton MINRES. We find that the eigenvalues
of the Hessian cluster and decay rapidly,
as has been observed by others
\cite{AlainRouxManzagol2019,GhorbanKrishnanXiaoi2019,SagunBottouLeCun2016},
suggesting that $k \ll d$ and $r \ll d$.
Numerical results demonstrate that, relative to first order
methods, these stochastic inexact Newton methods often converge
faster, are more cost-effective, and generalize better. To accompany this paper, we present \texttt{hessianlearn}, a Python library for second order stochastic optimization methods in TensorFlow, which can be found at  \url{https://github.com/tomoleary/hessianlearn}.

\section{Background}

Notation: For matrices $A$ and $B$m $A \succeq B$ means that $A - B$ is semi positive definite. We work in finite dimensional Hilbert spaces with inner product $ x^T y$, and corresponding norm $\| \cdot\| = \|\cdot \|_2$, or the Euclidean $\ell^2$ distance on vectors in $\mathbb{R}^d$. By $\mathbb{E}$ we mean the expectation taken against the measure $\nu$. When we say $\mathbb{E}_k$ we mean the conditional expectation at an iteration $k$ taken over all possible sample batches $X_k$. In the Euclidean space $\mathbb{R}^d$ we denote by $B_r(w)$ the ball of radius $r$ centered at $w$.

In solving \eqref{empirical_risk_minimization} one seeks to find a candidate solution $w^*$ that satisfies optimality conditions.
\begin{definition}[Stationary points]\label{first_order_optimality}
A point $w^*$ is a first order stationary point if $\|\nabla F(w^*)\| = 0$. A point $w^*$ is an $\epsilon$-first order stationary point if $\|\nabla F(w^*)\| < \epsilon$. A point $w^*$ is a second order stationary point if
\begin{equation}
  \|\nabla F(w^*)\| = 0 \quad \text{and} \quad 0 \preceq \nabla^2 F(w^*). 
\end{equation}
A point $w^*$ is an $(\epsilon_g,\epsilon_H)$-second order stationary point if
\begin{equation}\label{eps_second_order_stationary}
  \|\nabla F(w^*)\| \leq \epsilon_g \quad \text{and} \quad -\epsilon_H I \preceq \nabla^2 F(w^*), 
\end{equation}
for some $\epsilon_g,\epsilon_H >0$. A point $w^*$ is a stochastic $(\epsilon_g,\epsilon_H)$-second order stationary point if
\begin{equation}
  \mathbb{E}[\|\nabla F(w^*)\|] \leq \epsilon_g \quad \text{and} \quad -\epsilon_H I \preceq \mathbb{E}[\nabla^2F(w^*)].
\end{equation}
\end{definition}

We wish to solve the empirical risk minimization \eqref{empirical_risk_minimization} via the gradient based iteration \eqref{gradient_iteration}, approximately solving the Newton system
\begin{equation} \label{generic_newton_solve}
  \nabla^2 F_{S_k}(w_k)p_k = -\nabla F_{X_k}(w_k).
\end{equation}
The subsampled gradient is calculated over data $X_k$, while for computational economy the subsampled Hessian is calculated over data $S_k \subset X_k$, where $N_{S_k} \ll N_{X_k}$. Due to ill conditioning and nonconvexity we consider the Tikhonov regularized empirical risk minimization problem
 \begin{equation}\label{tikhnov_regularized_exp_risk_minimization}
   \min_w \overline{F}_X =   \frac{1}{N_{X}}\sum\limits_{i =1}^{N_{X}}F_i(w),  +\frac{\gamma}{2}\|w\|^2,
 \end{equation}
 for some $\gamma >0$ \cite{Tikhonov1963}. Much of the following can be extended to other regularizations such as $\ell^1$ or cubic regularization, however this is out of the scope of this work.

We now state some assumptions that will be used later in the paper. Assumptions \ref{assumption_a1}-\ref{assumption_a4} are adapted from \cite{BollapragadaByrdNocedal2018}.
\begin{enumerate}[label=A\arabic*]

  \item \label{assumption_a1} (Dominant positive eigenvalues) The function $F$ is twice continuously differentiable and any subsampled Hessian is spectrally bounded from above with constant $L$. That is, for any integer ${N_{S}}$ and set $S$ with $|S| = {N_{S}}$, there exists a positive constant $L_{N_{S}} < L$ such that 
  \begin{equation}
    \nabla^2 F_S (w) \preceq L_{N_{S}} I.
  \end{equation}
Moreover the first $r$ eigenvalues of $\nabla^2 F_S (w)$ evaluated along a path of iterates starting at $w_0$ are positive.

  \item \label{assumption_a2} (Bounded variance of sample gradients) There exists a constant $v$ such that
  \begin{equation}
    tr(\text{Cov}(\nabla F_i(w))) \leq v^2 \quad \forall w \in \mathbb{R}^d
  \end{equation}

  \item \label{assumption_a3} (Lipschitz continuity of Hessian) The Hessian of the objective function $F$ is Lipschitz continuous, i.e., there exists a constant $M>0$ such that 

  \begin{equation}
    \|\nabla^2F(w) - \nabla^2 F(z) \| \leq M\|w - z \|^2 \quad \forall w,z \in \mathbb{R}^d
  \end{equation}

  \item \label{assumption_a4} (Bounded variance of Hessian components) There exists $\sigma$ such that, for all component Hessians, we have
  \begin{equation}\label{bounded_var_hess_comp}
    \|\mathbb{E}[(\nabla^2 F_i(w) - \nabla^2F(w))^2]\| \leq \sigma^2, \quad \forall w \in \mathbb{R}^d
  \end{equation}

  \item \label{assumption_a5} ($\epsilon_g$-first order stationary point). For a given candidate stationary point $w^*$ and gradient batch size $N_{X_k}$, there exists $\epsilon_g>0$ such that
  \begin{equation}
    \mathbb{E}_k[\|\nabla F_{X_k} (w^*)\|] \leq \epsilon_g
  \end{equation}

\end{enumerate}

\section{Inexact Newton Methods} \label{inexact_newton_section}
 A subsampled inexact Newton method as described in \cite{DemboEisenstatSteihaug1982} is a method for which the Newton system \eqref{generic_newton_solve} is solved inexactly, and the linear solve is terminated when the following condition is satisfied:
\begin{equation} \label{inexact_forcing}
    \|\nabla^2 \overline{F}_{S_k}p_k + \nabla \overline{F}_{X_k}\| \leq \eta_k \|\nabla \overline{F}_{X_k}\|.
\end{equation}
When the gradient is large, the tolerance for inexactness is high. The tolerance tightens as one nears the solution. This avoids unnecessary work in the linear solves far from the solution, but still retains superlinear or quadratic convergence near the solution. Optimal choices of $\eta_k$ are discussed in the papers of Eisenstat and Walker \cite{DemboEisenstatSteihaug1982,EisenstatWalker1996}. We establish the following local convergence rate for a stochastic inexact Newton method for the choice of $\eta_k \leq \|\nabla \overline{F}_{X_k}\|$.

\begin{restatable}[Local convergence for stochastic inexact Newton methods with gradient norm forcing]{theorem}{inexacteisen}
\label{inexact_eisenprop}
Let $w^*$ be a stationary point and suppose that assumptions \ref{assumption_a1}-\ref{assumption_a4} hold, let 
\begin{equation}
  \mu = \min \bigg\{\|\nabla^2F(w^*) + \gamma I\|^{-1}, \|[\nabla^2 F(w^*) + \gamma I]^{-1}\| \bigg\},
\end{equation}
and assume that
\begin{enumerate}[label=(\alph{enumi})]
  \item $w_k \in B_{\delta}(w^*)$ with $\delta < \frac{2\mu}{L_{N_{S_k}} }$,
  \item $-\epsilon_H I \preceq \nabla^2 F_{S_k}$ for all $S_k$ and for all $w \in B_\delta(w^*)$,
  \item The Tikhonov regularization parameter is chosen such that $\gamma > \epsilon_H$,
  \item $\|\nabla^2\overline{F}_{S_k}(w_k)p_k - \nabla \overline{F}_{X_k}(w_k)\| \leq \eta_k\|\nabla \overline{F}_{X_k}(w_k)\|$ with $\eta_k \leq \|\nabla \overline{F}_{X_k}(w_k)\|$.
\end{enumerate}
Then for the iterate $w_{k+1} = w_k + \alpha_k p_k$, we have the following bound:
\begin{equation}\label{inexact_eisenprop_bound}
    \mathbb{E}_k[\|w_{k+1} - w^*\|] \leq c_0 +c_1 \|w_k - w^*\| + c_2 \|w_k - w^*\|^2,
\end{equation}
where
\begin{subequations}
\begin{align}
   c_0 &= \frac{1}{\gamma - \epsilon_H} \bigg[\frac{\alpha_k v}{\sqrt{N_{X_k}}}\bigg(1 + \frac{v}{\sqrt{N_{X_k}}}\bigg)\bigg]\\
   c_1 &= \frac{1}{\gamma - \epsilon_H}\bigg(L_{N_{S_k}} |1 - \alpha_k| + \frac{\sigma}{\sqrt{N_{S_k}}} + \frac{2\alpha_k v\mu}{\sqrt{N_{X_k}}}\bigg)\\
   c_2 &= \frac{1}{\gamma - \epsilon_H}\bigg(\frac{M}{2} + \alpha_k\mu^2 \bigg).
\end{align}
\end{subequations}

\end{restatable}

See Appendix \eqref{proof_of_inexact_eisenprop} for proof of Theorem \ref{inexact_eisenprop}. Assumption (a) states that $w_k$ is sufficiently close to an optimum. Assumption (b) states that eigenvalues of the Hessian are not too negative, and assumption (c) guarantees the Tikhonov regularized Hessian is invertible. Assumption (d) is the Eisenstat-Walker forcing condition. Ideally the constants $c_0,c_1$, and $c_2$ will be as small as possible. The constant $c_0$ will be small when the Monte Carlo approximation of the gradient is good. The constant $c_1$ will be small when the Monte Carlo approximation of the gradient and Hessian are both good, and the full Newton step $\alpha_k=1$ can be taken. The constant $c_2$ will be small when the Hessian is well conditioned.

This theorem says nothing about how expensive the method may be per iteration. The per iteration cost of the method will depend on the spectral properties of the Hessian and the batch sizes. In the next subsections, we will analyze how solving the Newton system approximately, rather than exactly, affects the convergence rate.

\subsection{Low Rank Newton Methods}
Let the spectral decomposition of the Hessian be given as follows:
\begin{equation}
  \nabla^2 F = U\Lambda U^T = \sum_{i=1}^d \lambda_i u_i u_i^T,
\end{equation}
where the eigenvalues $\lambda_i$ are sorted such that $|\lambda_i| \geq |\lambda_j|$ for all $i >j$, and $u_i\in \mathbb{R}^d$ are the corresponding eigenvectors. The objective function is most sensitive to perturbations to $w$ in directions corresponding to eigenvalues of large magnitude, because the energy landscape has large curvature in these directions. The objective function is least sensitive to perturbations of $w$ in directions corresponding to eigenvalues of small magnitude, because the energy landscape is approximately flat in these directions. The magnitude of the eigenvalue $\lambda_i$ is related to how informative the data are to the component of $w$ in the $u_i$ direction; the larger the eigenvalue, the more information one can learn about the parameter in the associated eigenvector direction \cite{Alger2019}. It is therefore important to resolve the subspace spanned by the dominant eigenvectors, whereas resolving the complementary subspace spanned by non-dominant eigenvectors is less important. We therefore consider low rank approximations to the Hessian,
\begin{equation}
  H^{(r)} = [\nabla^2 F]^{(r)} = U_r \Lambda_r U^T_r = \sum_{i=1}^r \lambda_i u_i u_i^T.
\end{equation}

The matrix $H^{(r)}$ has at most rank $r$, so the system $H^{(r)} p = -g$ is degenerate. For any $\gamma \neq -\lambda_i$, the matrix $H^{(r)} + \gamma I$ is invertible. This matrix arises in the case of Levenberg-Marquardt or trust region methods \cite{NocedalWright2006}, or Tikhonov regularization.
\begin{equation} \label{levenberg_tikhonov}
  [H^{(r)} +\gamma I]p_k = -\widehat{g}_k = -\begin{cases}
                            \nabla F & \quad \text{Levenberg-Marquardt}\\
                            \nabla F + \gamma w_k & \quad \text{Tikhonov regularization}
                        \end{cases}
\end{equation}
Equation \eqref{levenberg_tikhonov} can be solved efficiently using the Sherman-Morrison-Woodbury formula:
\begin{equation}\label{low_rank_newton}
        p_k = - \bigg[\frac{1}{\gamma}I_d - \frac{1}{\gamma^2}U_r \bigg(\Lambda_r^{-1} + \frac{1}{\gamma}I_r\bigg)^{-1}U_r^T \bigg]\widehat{g}_k.
\end{equation}
This method therefore interpolates between gradient descent and Newton's method. In $\text{span}\{U_r\}$, the regularized Newton direction is used. In $\text{span}\{U_r\}^\perp$, the regularization preconditioned gradient direction is used.

\subsubsection{Local convergence for the stochastic low rank Newton method}

Suppose that for each $w_k$ and $S_k$, we have the truncated eigenvalue decomposition $H^{(k)}_r  =[\nabla^2 F_{S_k}]_r= U^{(k)}_r \Lambda_r U^{(k)T}_r$ for the empirical risk function, and the iterates
\begin{equation} \label{trunc_newton_update}
  w_{k+1} = w_k - \alpha_k [H_k^{(r)} + \gamma I]^{-1}\nabla \overline{F}_{X_k}(w).
\end{equation}

What is needed are methods that do not require the explicit matrix, and do not need many matrix vector products. Randomized SVD can be used to determine the numerical rank of the Hessian and compute low rank factorizations of the Hessian efficiently without explicitly forming it, using just Hessian-vector products \cite{HalkoMartinssonTropp2011}. Lanczos procedures can also be used for matrix-free low rank factorization \cite{AlainRouxManzagol2019,DauphinPescanuGulcehre2014}, but we prefer randomized SVD because of its robustness to repeated eigenvalues \cite{HalkoMartinssonTropp2011}. We now state a bound for the convergence in expected value of stochastic low rank Newton methods based on SVD.

\begin{restatable}[Local convergence of stochastic low rank Newton]{theorem}{lowrankconv}
\label{lowrankconvergence}
Let $\{w_k\}$ be the iterates generated by \eqref{trunc_newton_update}, let $w^*$ be a stationary point and suppose that assumptions \ref{assumption_a1} - \ref{assumption_a4} hold, then for each $k$
\begin{equation}
    \mathbb{E}_k[\|w_{k+1} - w^*\|] \leq c_0 + c_1 \|w_k - w^*\| + c_2 \|w_k - w^*\|^2,
\end{equation}
where
\begin{subequations}
\begin{align}
  c_0 &=  \frac{\alpha_k v}{|\overline{\lambda_{r}^{(k)}}+\gamma|\sqrt{N_{X_k}}} ,\\
  c_1  &= \frac{1}{|\overline{\lambda_{r}^{(k)}}+\gamma|}\bigg[L_{N_{S_k}}|1- \alpha_k| + \mathcal{E}|\overline{\lambda_{r+1}^{(k)}}| + \gamma + \frac{\sigma}{\sqrt{N_{S_k}}} \bigg],\\
  c_2 &= \frac{M}{2|\overline{\lambda_{r}^{(k)}}+\gamma|}.
\end{align}
\end{subequations}
Here we define $\overline{\lambda_{r}^{(k)}} = \mathbb{E}_k[\lambda_{r}^{(k)}]$. The error coefficient $\mathcal{E} = 1$ when the truncated low rank spectral decomposition is exact, and $\mathcal{E} = \bigg(1+4\frac{\sqrt{d(r+p)}}{p-1}\bigg)$ in the case that it is calculated using randomized SVD.

\end{restatable}

See Appendix \ref{low_rank_appendix} for proof. For fast convergence, we want the constants $c_0,c_1$ and $c_2$ to be as small as possible. The constant $c_0$ is small when the Monte Carlo approximation error for the gradient is small. The constant $c_1$ has errors from step length (if $\alpha \neq 1$), the Hessian Monte Carlo approximation, the low rank Hessian approximation, and in the case of randomized SVD, the additional approximation factor $\mathcal{E}$. When the Hessian has low rank, the approximation error by low rank factorization will be small. The Hessian is often low rank in machine learning applications \cite{AlainRouxManzagol2019,GhorbanKrishnanXiaoi2019,SagunBottouLeCun2016}. We also observe that the Hessian has low rank in our numerical experiments. The constant $c_2$ will be small when the Hessian is well conditioned.


\subsubsection{Low Rank Saddle Free Newton}

If the Hessian is invertible without regularization, exact Newton rescales the negative gradient component-wise in the Hessian eigenbasis, by the corresponding eigenvalue,
\begin{equation}
  p = -[\nabla^2F]^{-1}\nabla F = -\sum_{i=1}^d \frac{1}{\lambda_i}(\nabla F^Tu_i)u_i.
\end{equation}
When an eigenvalue is negative, the components of the gradient in this direction will change sign and point towards the saddle point, instead of away. Therefore exact Newton iterates may converge to saddle points. One remedy to this is the SFN algorithm, in which negative eigenvalues of the Hessian are flipped to be positive \cite{DauphinPescanuGulcehre2014,GillMurrayWright1981}. In the SFN method, one solves $|\nabla^2 F|p = -\nabla F$, where $|\nabla F| = U|\Lambda|U^T$. We propose a low rank version of this, using the Sherman-Morrison-Woodbury formula.

\begin{algorithm}[H]\label{lrsfnewton}
\SetAlgoLined
Given $w_0$\\
\While{ not converged }{
  \If{$\|\nabla \overline{F}_{X_k}\| \leq \epsilon_g$ and $\lambda_\text{min}(\nabla^2 F_{S_k}) \geq - \epsilon_H$}{
  break
  }
  Calculate $U^{(k)}_r \Lambda_r U^{(k)T}_r \approx H^{(r)}$ via randomized SVD \cite{HalkoMartinssonTropp2011}\\
  Calculate $p_k = - \bigg[\frac{1}{\gamma}I_d - \frac{1}{\gamma^2}U_r^{(k)} \bigg(|\Lambda^{(k)}_r|^{-1} + \frac{1}{\gamma}I_r\bigg)^{-1}U^{(k)T}_r \bigg]\nabla \overline{F}_{X_k}$\\
  $\alpha_k$ given or computed via line search\\
  $w_{k+1} = w_k + \alpha_k p_k$
}
\caption{Randomized Low Rank Saddle Free Newton}
\end{algorithm}

When the Hessian is positive definite in $\text{span}(U^{(k)}_r)$ ($\lambda_i >0$ for $i \leq r$), this method is identical to randomized low rank Newton algorithm. In this case convergence will be identical to that of Theorem \ref{lowrankconvergence}. The low rank saddle free Newton method is designed to escape indefinite regions that low rank Newton or other methods may stagnate in, by incorporating directions of negative curvature.


\subsection{Inexact Newton-Krylov Methods}

Krylov methods are the preferred linear solver for inexact Newton methods. In this section we discuss their extension to stochastic nonconvex problems.
\begin{definition}[Krylov Subspace] \label{krylov_space}
Given $A:\mathbb{R}^d \rightarrow \mathbb{R}^d$ and $y \in \mathbb{R}^d$, we define the $m^\text{th}$ Krylov subspace as the linear subspace $\mathcal{K}_m(A,y) \subset \mathbb{R}^d$
\begin{equation}
  \mathcal{K}_m(A,y) = \text{span}\{y,Ay,\dots,A^{m-1}y\}.
\end{equation}
\end{definition}
Given $p_0 =0$ as an initial guess, stochastic Newton-Krylov methods approximate
\begin{equation}
  p = -[\nabla^2\overline{F}_{S_k}]^{-1}\nabla \overline{F}_{X_k} \approx p_m \in \mathcal{K}_m(\nabla^2\overline{F}_{S_k},-\nabla \overline{F}_{X_k})
\end{equation}
via a Galerkin projection. Similar to randomized low rank methods, Krylov methods require only the action of a matrix on vectors; access to the entries of the matrix is not required. In this work the Krylov methods we consider are conjugate gradients (CG), the minimal residual method (MINRES), and the generalized minimal residual method (GMRES). GMRES applies to indefinite matrices, MINRES applies to symmetric indefinite matrices. CG can be adapted to symmetric indefinite matrices by a simple modification. A generic stochastic inexact Newton-Krylov method is described below.

\begin{algorithm}[H]\label{inkrylov}
\SetAlgoLined
Given $w_0$\\
\While{ not converged }{
  \If{$\|\nabla \overline{F}_{X_k}\| \leq \epsilon_g$ and $\lambda_\text{min}(\nabla^2 F_{S_k}) \geq - \epsilon_H$}{
  break
  }
  Given $\|\nabla \overline{F}_{X_k}\|$ compute $\eta_k$ via Eisenstat-Walker\\
  Solve $\|\nabla^2 \overline{F}_{S_k}p_k + \nabla \overline{F}_{X_k}\| \leq \eta_k \|\nabla \overline{F}_{X_k}\|$ via a Krylov method\\
  $\alpha_k$ given or computed via line search\\
  $w_{k+1} = w_k + \alpha_k p_k$
}
\caption{Inexact Newton-Krylov Methods}
\end{algorithm}

\subsubsection{Local convergence rates}

In the case that CG is used for the linear solve, Bollapragada et al.\ have derived a local convergence rate for the \textit{semi-stochastic} case, in which the gradient is not subsampled \cite{BollapragadaByrdNocedal2018}. We extend this analysis to the fully stochastic setting, including the dependence of the convergence constants on the parameters $\alpha_k$ and $\gamma$.

\begin{restatable}[Local convergence of stochastic inexact Newton CG (INCG), extension of Lemma 3.1 of \cite{BollapragadaByrdNocedal2018}]{theorem}{incgconvergencerestat}
\label{incg_convergence}

Let $w^*$ be a stationary point, suppose assumptions \ref{assumption_a1}-\ref{assumption_a4} hold, and the iterates $\{w_k\}$ are generated by the stochastic inexact Newton CG method, the direction $p_k^r$ is found in $r \ll d$ steps (for justification see section \ref{super_eigen_approx}), and there exists $\epsilon_H>0$ such that $-\epsilon_H I \preceq \nabla^2 F_{S_k}(w_k)$ and $\gamma >\epsilon_H$. Then, 

\begin{equation} \label{incg_bound_eq}
  \mathbb{E}_k[\|w_{k+1} - w^*\|] \leq c_0 + c_1 \|w_k - w^*\| + c_2 \|w_k - w^*\|^2
\end{equation}
where
\begin{subequations}
\begin{align}
  c_0 &= \frac{\alpha_k v}{(\gamma - \epsilon_H)\sqrt{N_{X_k}}} \\
  c_1 &=  \frac{1}{\gamma - \epsilon_H}\bigg[ L_{N_{S_k}}|1 - \alpha_k| + \frac{\sigma}{\sqrt{N_{S_k}}} +2\alpha_k L_{N_{S_k}}\sqrt{\kappa_{N_{S_k}}} \bigg(\frac{\sqrt{\kappa_{N_{S_k}}} - 1}{\sqrt{\kappa_{N_{S_k}}} + 1 }\bigg)^r \bigg]  \\
  c_2 &= 
  \frac{M}{2(\gamma - \epsilon_H)}, 
\end{align}
and $\kappa_{N_{S_k}}$ is the condition number of the Tikhonov regularized Hessian.
\end{subequations}

\end{restatable}

See Appendix \ref{proof_of_incg_convergence} for proof. For fast convergence, we want the constants $c_0,c_1$ and $c_2$ to be as small as possible. The constant $c_0$ is small when the Monte Carlo approximation error for the gradient is small. The term $c_1$ will be small when the full Newton step ($\alpha_k=1$) can be taken, when the Monte Carlo approximation of the Hessian is accurate, and the linear solve error is small after $r$ steps of CG. The constant $c_2$ will be small when the Hessian is well conditioned.

\begin{remark}
In each of the theorems up to this point (Theorem \ref{inexact_eisenprop}, Theorem \ref{lowrankconvergence} and Theorem \ref{incg_convergence}), the constant $c_0$ depends only on the Monte Carlo error in the gradient calculation. In the semi-stochastic setting each of these convergence rates is then linear-quadratic. These bounds can then be used to derive conditions for super-linear convergence in the semi-stochastic case, as in \cite{BollapragadaByrdNocedal2018}.
\end{remark}

Worst case bounds for Krylov method convergence can often be established based on the condition number of the matrix. The convergence of Krylov methods will more generally depend on the entire spectrum of the Hessian, and will benefit from spectral clustering, as we will see in the next section. Preconditioners can be used to enhance convergence by reducing the condition number of the matrix, or clustering the spectrum. In the case of MINRES, worst case bounds can be established based on the condition number. GMRES achieves superior convergence if the spectrum of $H$ resides in an interval that does not include the origin. We present the following result about the convergence rate of stochastic inexact Newton GMRES and MINRES algorithms (INGMRES and INMINRES).

\begin{restatable}[Local convergence of stochastic inexact Newton GMRES and MINRES]{theorem}{gmresconvrestat}
\label{gmres_convergence}
Let $w^*$ be a stationary point, suppose that assumptions \ref{assumption_a1}-\ref{assumption_a5} hold and we have that additionally for some $\delta >0$, $-\epsilon_H I \preceq \nabla^2 F_{S_k}$ for all $S_k$ and for all $w \in B_\delta(w^*)$ and $\gamma > \epsilon_H $, and the direction $p_k^r$ is found in $r \ll d$ steps. Then we have the following expected value convergence rate bound for the iterates of the stochastic inexact Newton GMRES/MINRES methods:
\begin{equation} \label{minres_gmres_bd_eq}
  \mathbb{E}_k[\|w_{k+1} - w^*\|] \leq c_0 + c_1 \|w_k - w^*\| + c_2 \|w_k - w^*\|^2
\end{equation}
where
\begin{subequations}
\begin{align}
  c_0 &= \frac{\alpha_k }{\gamma - \epsilon_H} \bigg(\frac{ v}{\sqrt{N_{X_k}}} + \frac{ \epsilon_g }{(\gamma - \epsilon_H)}\mathcal{E} \bigg)\\
  c_1 &= \frac{1}{\gamma - \epsilon_H}\bigg(L_{N_{S_k}}|1 - \alpha_k| +  \frac{\sigma}{\sqrt{N_{S_k}}} + \frac{\alpha_k L_{N_{X_k}}}{(\gamma - \epsilon)}\mathcal{E}\bigg) \\
  c_2 &= 
  \frac{M}{2(\gamma - \epsilon_H)}. 
\end{align}
\end{subequations}
When the GMRES solver is used,
\begin{subequations}
\begin{align}
  &\mathcal{E} = \frac{L_{N_{S_k}}}{\gamma - \epsilon_H} \frac{C_r(\frac{a}{d})}{|C_r(\frac{c}{d})|} \\
  a = (L_{N_{S_k}} - \gamma + \epsilon_H) +2\epsilon, \quad &c = \frac{1}{2}(L_{N_{S_k}} + \gamma - \epsilon_H), \quad d = \frac{1}{2}(L_{N_{S_k}} - \gamma + \epsilon_H)
\end{align}
\end{subequations}
and $C_r$ is the $r^{th}$ order Chebyshev polynomial. For MINRES,
\begin{equation}
  \mathcal{E} = \bigg(1 - \frac{(\gamma - \epsilon_H)^2}{L_{N_{S_k}}^2}\bigg)^\frac{r}{2}.
\end{equation}
\end{restatable}

See Appendix \eqref{proof_of_gmres_convergence} for proof of Theorem \ref{gmres_convergence}. The constant $c_0$ will be small when the Monte Carlo error for the gradient is small, and the approximation of the linear solve via GMRES is accurate. The constant $c_1$ will be small when the full Newton step $\alpha_k=1$ can be taken, the Monte Carlo error for the Hessian is small, and the approximation of the linear solve via GMRES is accurate. The constant $c_2$ will be small when the Hessian is well conditioned.

\begin{remark}
In Theorem \ref{gmres_convergence} the constant $c_0$ depends not only on the Monte Carlo approximation of the gradient, but also on the error in the Krylov solve $\mathcal{E}$ and the constant $\epsilon_g$ from Assumption A5. In order to derive a linear-quadratic convergence rate from this bound in the semi-stochastic case, one needs to employ the restrictive assumption that $\epsilon_g =0$, i.e. $w^*$ is a local minimum for all of the sample gradients.
\end{remark}

\subsubsection{Superior Approximation for Clustered Eigenvalues} \label{super_eigen_approx}

Krylov subspaces are intimately related to spaces of polynomials. The Krylov subspace $\mathcal{K}_m(A,y)$ is the space of all vectors $x \in \mathbb{R}^d$ that can be written as $x =p(A)y$ where $p \in \mathbb{P}_{m-1}$, the space of all polynomials of degree $m-1$ or less.  If the generating vector $y$ is not degenerate---in the basis of eigenvectors of $A$, none of its components are zero---then there is a natural isomorphism between $\mathbb{P}_{m-1}$ and $\mathcal{K}_m$ defined by \cite{Saad2003}:
\begin{equation}
  \mathbb{P}_{m-1} \ni q \mapsto x = q(A)y \in \mathcal{K}_m(A,y).
\end{equation}

This isomorphism with polynomials allows us to analyze Krylov solution by instead analyzing polynomials. Let $A = U\Lambda U^T$ denote the eigenvalue decomposition of $A$, with $\Lambda = \text{diag}(\lambda_k)$ and $u_k$ is the $k^{th}$ column of $U$. Let $\mathcal{Q}_m$ denote the set of all $m^{th}$ order polynomials with constant term $1$, that is
\begin{equation}
    \mathcal{Q}_m = \{q \in \mathbb{P}_m | q(0) = 1 \}.
\end{equation}
We have the following well-known results for CG, GMRES and MINRES.

\begin{theorem}[Krylov methods and minimum polynomials] \label{krylov_polynomial}
Let $x_m$ be the $m^{th}$ CG iterate for solving $Ax = b$, and $x^* = A^{-1}b$ . The following relationship holds:
\begin{equation}\label{cg_polynomial_eq}
    \|x^* - x_m \|^2_A = \min_{q\in \mathcal{Q}_m}\sum_{k=1}^n \lambda_k q(\lambda_k)^2(u_k^Te_0)^2
\end{equation}
where $e_0 = x^* - x_0$.
Let $x_m$ be the $m^{th}$ GMRES (with no restarts) or MINRES iterate for solving $Ax = b$.  Then the following relationship holds:
\begin{equation} \label{gmres_polynomial_eq}
  \|b - Ax_m\|^2 = \min_{q \in \mathcal{Q}_m} \sum\limits_{k=1}^nq(\lambda_k)^2 (u_k^Tr_0)^2
\end{equation}
where $r_0 = b -Ax_0$ is the initial residual.

\end{theorem}

These results are canonical; a proof of the result regarding CG is given by Shewchuck \cite{Shewchuk1994}. The GMRES/MINRES result follows from a similar argument. When eigenvalues are clustered, a lower degree polynomial is better able to minimize either \eqref{cg_polynomial_eq} or \eqref{gmres_polynomial_eq}. This means that CG, GMRES, and MINRES will perform better than low rank approximations when the eigenvalues are clustered. Due to the additional $\lambda_k$ in equation \eqref{cg_polynomial_eq}, CG will eliminate errors in the subspaces corresponding to large eigenvalues more aggressively than in the subspaces corresponding to small eigenvalues (GMRES and MINRES do not have this property).

\subsubsection{Newton-Krylov methods and saddle points} \label{newtonkrylovconditioning}

Inexact Newton-Krylov methods can be made robust to saddle points. At a given Krylov iteration with a search direction update $v_m$, an un-normalized Rayleigh quotient $v_m^THv_m$ can be calculated (in the case of CG this term is already calculated). If this quantity is negative, then the vector $v_m$ points in a direction of negative curvature. One can then terminate the Krylov solve early, without modifying the search direction $p_m$. The drawback of the early termination procedure relative to saddle free Newton is that without the spectral decomposition which yields explicit knowledge of the negative curvature direction, one cannot accelerate the escaping of indefinite regions.

\subsection{Comparing costs: gradient vs. Hessian}\label{comparing_costs}

So far the development has been based on the finite sum optimization problem \eqref{empirical_risk_minimization}. We restrict the discussion at this point to neural network training. The dominant costs associated with neural network training are the evaluations of the neural network and its derivatives. The gradient can be formed efficiently using an adjoint process 
(referred to as back propagation in the neural network literature), which amounts to a forward and backward evaluation of the neural network \cite{RumelhartHintonWilliams1988}. The action of the Hessian on a vector can be formed using an adjoint based method, by an additional forward and backward evaluation of the neural network as described by Pearlmutter \cite{Pearlmutter1994}. We refer to the pair of one forward and one backward evaluation of the neural network as a \textit{sweep}.\footnote{Note that the forward evaluation for the gradient will typically be nonlinear, while the backward evaluation for the gradient and the forward and backward evaluations for the Hessian-vector product will be affine since they involve the transpose of the Jacobian of the forward mapping in the case of the gradient, and similar terms for the Hessian. We count these sweeps all the same, even though the Hessian sweeps may be cheaper.} 

For a given iteration of a low rank Newton method with rank $r$ and oversampling parameter $p$, the number of network sweeps used to construct the low rank Hessian approximation can be expressed as follows:
\begin{equation}\label{lowranksweeps}
  \# (\text{Low rank Hessian sweeps}) = 2C(r+p) N_{S_k} .
\end{equation}
Here $C=1,2$ depending on if single pass or double pass algorithms are used for randomized SVD \cite{HalkoMartinssonTropp2011}. We use the double pass algorithm. The total neural network sweeps for the double pass version of the low rank SFN algorithm, including the cost of computing the gradient, is then
\begin{equation}\label{lowranknewtonsweeps}
  \#(\text{Low rank Newton sweeps}) = \bigg(N_{X_k}+  4(r+p) N_{S_k}\bigg).
\end{equation}
The cost of the associated linear algebra for randomized SVD will yield an additional $O(dr^2 + r^3)$ operations. For the inexact Newton-Krylov method with $r$ Krylov iterations,

\begin{equation}\label{inexactsweeps}
  \#(\text{Inexact Newton-Krylov sweeps}) = \bigg(N_{X_k}+ 2r N_{S_k}\bigg).
\end{equation}

Previous analysis (Theorem \ref{inexact_eisenprop}, Theorem \ref{lowrankconvergence}, and Theorem \ref{incg_convergence}, Theorem \ref{gmres_convergence}) suggests that taking $N_{X_k}$ large is important if one desires fast convergence. As for $N_{S_k}$, Bollapragada et al.\ use convergence rates similar to the ones presented in previous analysis to derive conditions on how to increase $N_{S_k}$ to maintain superlinear convergence rates \cite{BollapragadaByrdNocedal2018}. Since the computational cost will grow with this increase in batch size, we take $N_{S_k} \ll N_{X_k}$ fixed. How large then, should $N_{S_k}$ be, so that the subsampled Hessian is representative of the true Hessian? Xu et al.\ \cite{XuRoostaMahoney2017a} use the Operator-Bernstein inequality to derive probabilistic spectral convergence of the uniformly subsampled Hessian, which can be used to guide the choice of $N_{S_k}$.

\begin{theorem}[Complexity of Uniform Sampling, Lemma 4 Xu et al.\ \cite{XuRoostaMahoney2017a}]
Given $0 < \epsilon,\delta < 1$, and assume assumption A1 holds, let
\begin{equation}
  N_{S_k} \geq \frac{16L^2}{\epsilon^2} \log\frac{2d}{\delta}
\end{equation}
at any $w \in \mathbb{R}^d$, and suppose that the elements of $S_k$ are chosen uniformly at random, with or without replacement from $X$. Then the subsampled Hessian obeys the probabilistic bound
\begin{equation}
  \text{Pr}\bigg(||\nabla^2 F_{S_k}(w) - \nabla^2 F(w)|| \leq \epsilon \bigg) \geq 1 - \delta.
\end{equation}
\end{theorem}

In the next section, we will numerically show that the subsampled Hessian still provides a good approximation of the true Hessian, even when $N_{S_k}$ is small relative to $N_{X_k}$. This empirical observation in combination with \eqref{lowranknewtonsweeps} and \eqref{inexactsweeps} suggests that the per iteration cost of a stochastic Newton method is not substantially more than the per iteration cost of gradient descent. But stochastic Newton methods will have superior convergence properties.




\section{Numerical Experiments}\label{numerical_results}

In various computer vision problems such as image classification, convolutional autoencoders are used to learn a low dimensional compressed representation of an image. We consider two autoencoder training problems: MNIST and CIFAR10 \cite{KrizhevskyNairHinton2010,LeCunCortes2010}. We study the convergence and generalization properties of low rank saddle free Newton (LRSFN), inexact Newton CG with early termination (INCG), inexact Newton GMRES (INGMRES), and inexact Newton MINRES (INMINRES). We compare these Newton methods with standard deep learning methods such as gradient descent (GD), stochastic gradient descent (SGD) and Adam \cite{KingmaBa2014}. In addition to studying the convergence and generalization properties of these algorithms, we also study the dominant spectrum of the Hessian along the sequence of iterates generated by INCG. For each problem, we run the optimizers for a fixed number of neural network sweeps, and compare the testing and training errors.

For the semi-stochastic Newton methods and gradient descent, we fix $X_k$ across iterations. In the fully stochastic Newton methods, as well as Adam and SGD, mini-batches of size $0.1N_{X_k}$ are subsampled from a fixed set $X$ at each iteration, we denote these methods SA for stochastic approximation. For all of the Newton methods, $S_k$ is subsampled from $X_k$. For the Newton-Krylov methods and GD, a line search is performed at each iteration, and when the sufficient descent condition is not met, the step with the corresponding $\alpha_k<1$ is taken anyways. This allows for the escape from suboptimal basins as in nonmonotonic line searches; it is justified experimentally, as the empirical risk typically goes down the iteration after the step is taken. The neural network evaluations associated with line search are also counted. 

In order to obtain numerical results robust to outliars, we run numerical experiments with various different initial guesses $w_0 \sim \mathcal{N}(0,I)$. For each run, we compute minimum testing and training error, which we denote $\widehat{F}_k$, and we report sample average statistics for training and testing errors. We take regularization parameter $\gamma =0.1$. Run specifications for each data set are summarized in Table \ref{conv_auto_run_specs}.

\begin{table}[H]
\begin{center}
  \begin{tabular}{|l||c|c|}
    \hline

     \enskip &  MNIST & CIFAR10 \\
    \hline \hline
    $N_{X_k}$ train &  10,000 & 10,000    \\ \hline
    $N_{X_k}$ train, SA &  1,000 & 1,000    \\ \hline
    $N_{S_k}$ train &  1,000 & 1,000    \\ \hline
    $N_{S_k}$ train, SA &  100 & 100    \\ \hline
    $N_{X_k}$ test &  10,000 & 10,000    \\ \hline
    \texttt{n\_filters} & [4,4,4,4] & [4,4,4,8,8,4,4,4]  \\ \hline
    \texttt{filter\_sizes} & [8,4,4,8] & [16,8,8,4,4,8,8,16]  \\ \hline
    \text{activation} & \text{softmax} & \text{softmax} \\ \hline
    $d$ & 517 & 12,315 \\ \hline
    $|x|$ & $(28)^2$ & $3(32)^2$ \\ \hline
    Max sweeps & 1,000,000 & 500,000  \\ \hline
  \end{tabular}
\end{center}
\caption[]{Summary of run specifications for convolutional autoencoders}
\label{conv_auto_run_specs}
\end{table}

We train four layer convolutional autoencoders for the MNIST problem and 8 layer convolutional autoencoders for the CIFAR10 problem. Autoencoders are typically harder to train as they get deeper, beyond that the CIFAR10 data set consists of larger images with more information content. The CIFAR10 problem is the harder problem for this reason. Each convolutional kernel (filter) is square, and the dimensions and number of filters for each layer are denoted in Table \ref{conv_auto_run_specs}. The regularity of the underlying optimization problem depends on the regularity of the Hessian (Assumption A1), which in turn depends on the derivative regularity of the activation functions via the chain rule. For this reason we only consider smooth functions such as softmax.

\subsection{MNIST}
We begin by observing how the choice of $N_{S_k}$ effects convergence for INCG and LRSFN for the MNIST dataset. As discussed in Section \ref{comparing_costs} the computational economy of the Newton methods hinges on $N_{S_k} \ll N_{X_k}$. 

In Figure \ref{mnist_s_k_conv} each run started from the same initial guess. For INCG the runs converged to many different basins, whereas the iterates generated by LRSFN runs follow paths resulting in similar empirical risk values. For the INCG runs, more Hessian information did not necessarily lead to better numerical results; to the contrary the methods often performed worse for larger Hessian batch sizes. The rank of the Hessian should be approximately non-decreasing for increasing batch sizes, so perhaps the conditioning gets worse as the Hessian batch size increases. If that is the case, then the problem could be alleviated by a preconditioner, however the cost of each CG iteration goes up linearly with the batch size so there is still incentive to keep $N_{S_k} \ll N_{X_k}$.

Since LRSFN only uses Hessian information in the dominant Hessian subspace of fixed rank it did not have the convergence of the method effected much by the change in the Hessian batch size. Indeed each run performed about exactly the same in terms of Newton iterations. This suggests that the dominant Hessian subspace may be well approximated for reasonably small $N_{S_k}$, which would imply a lot of redundancy in the information contained in $X_k$.

\begin{figure}[H] 
\centering
\begin{subfigure}[a]{0.5\textwidth}
   \includegraphics[width=\textwidth]{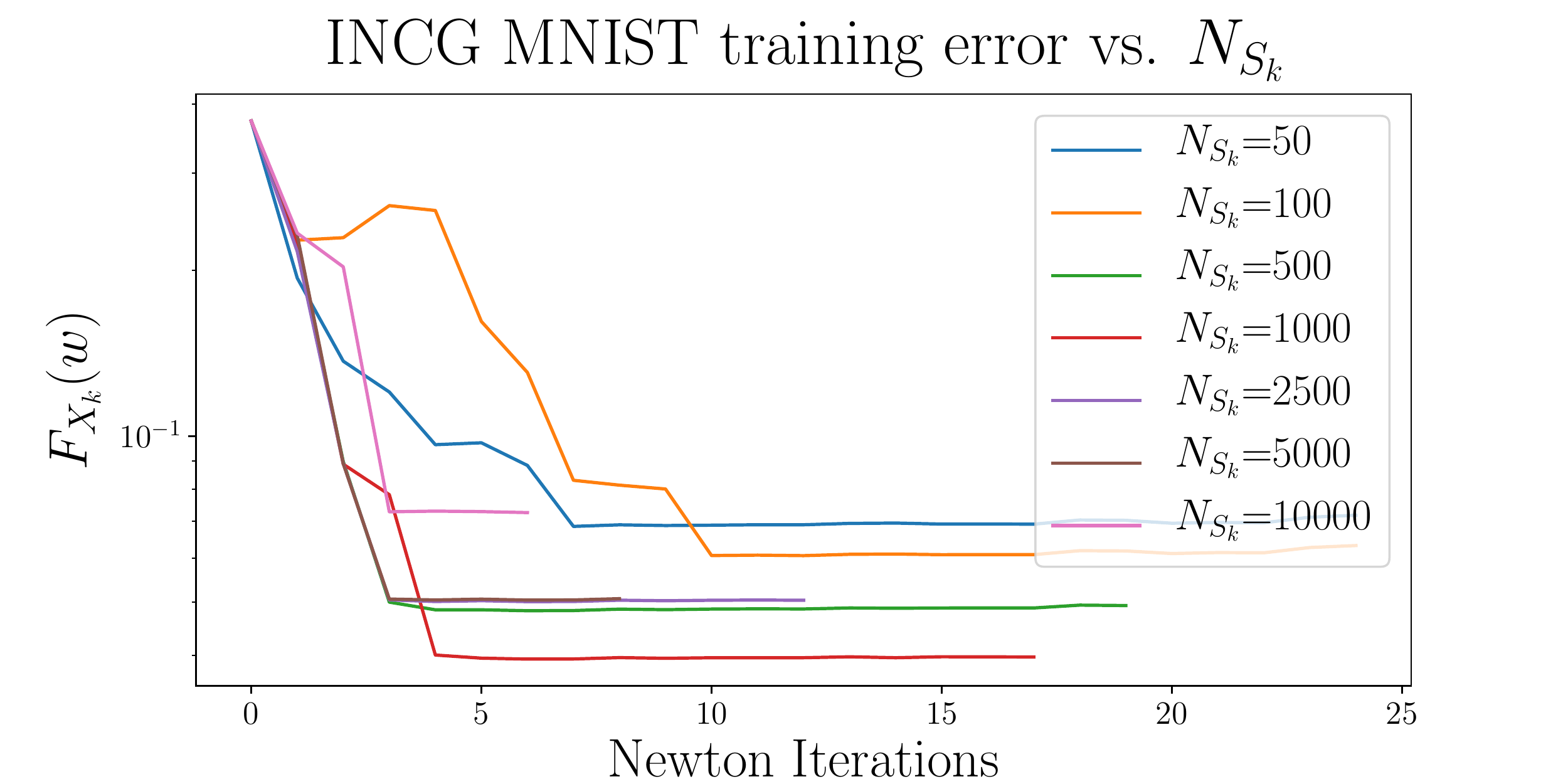}
\end{subfigure}%
\begin{subfigure}[a]{0.5\textwidth}
   \includegraphics[width=\textwidth]{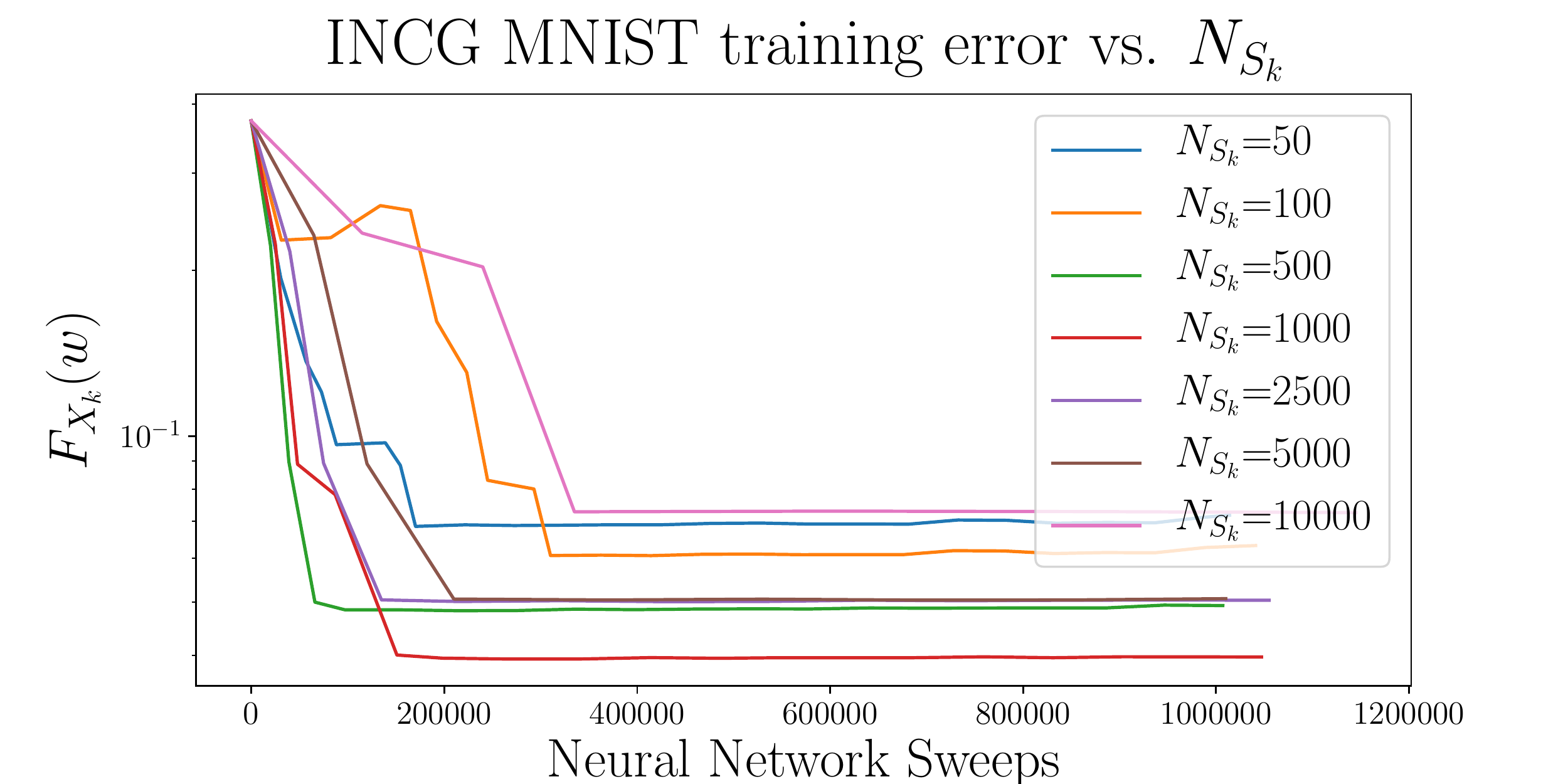}
\end{subfigure}
\begin{subfigure}[b]{0.5\textwidth}
   \includegraphics[width=\textwidth]{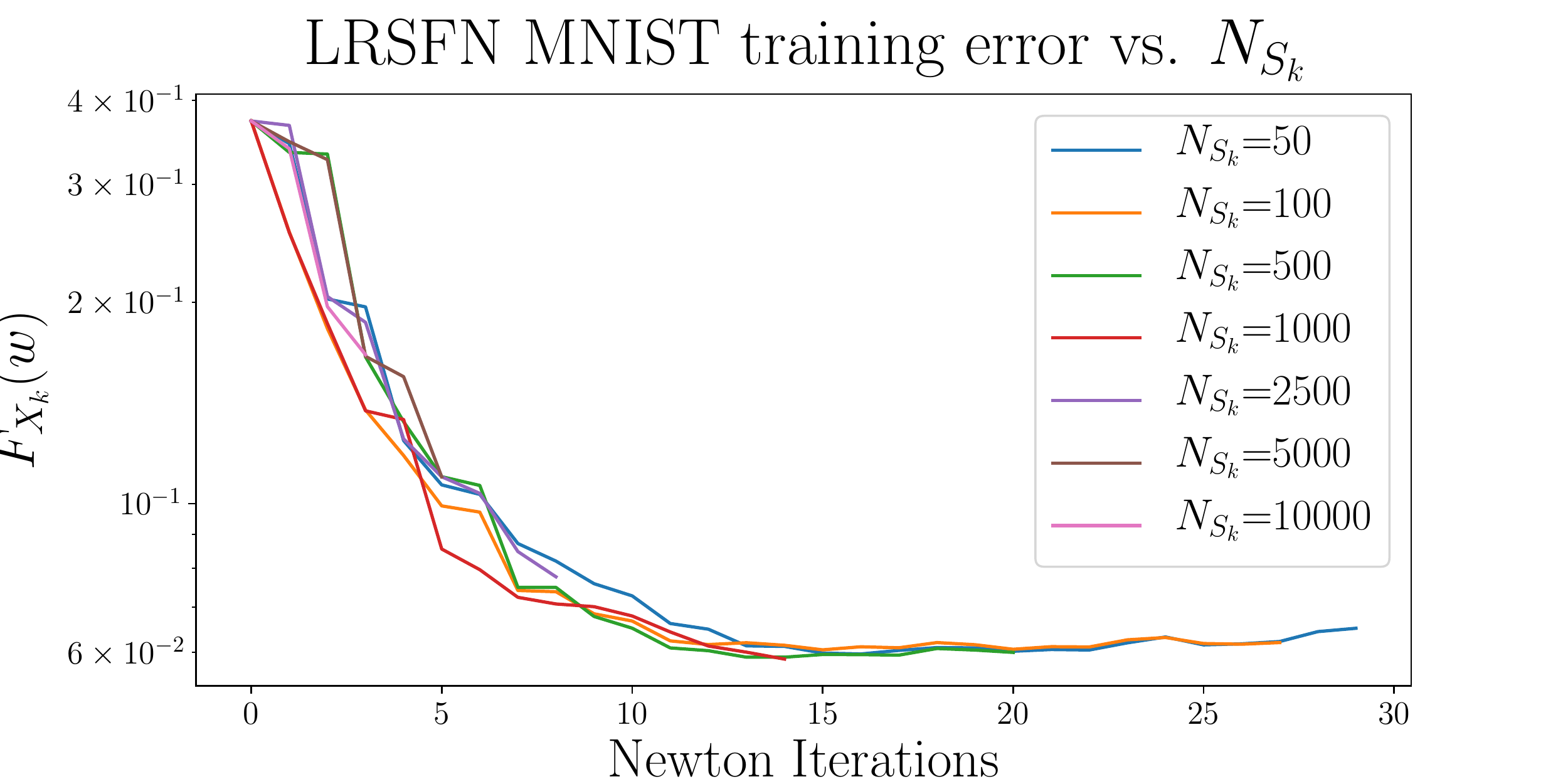}
\end{subfigure}%
\begin{subfigure}[b]{0.5\textwidth}
   \includegraphics[width=\textwidth]{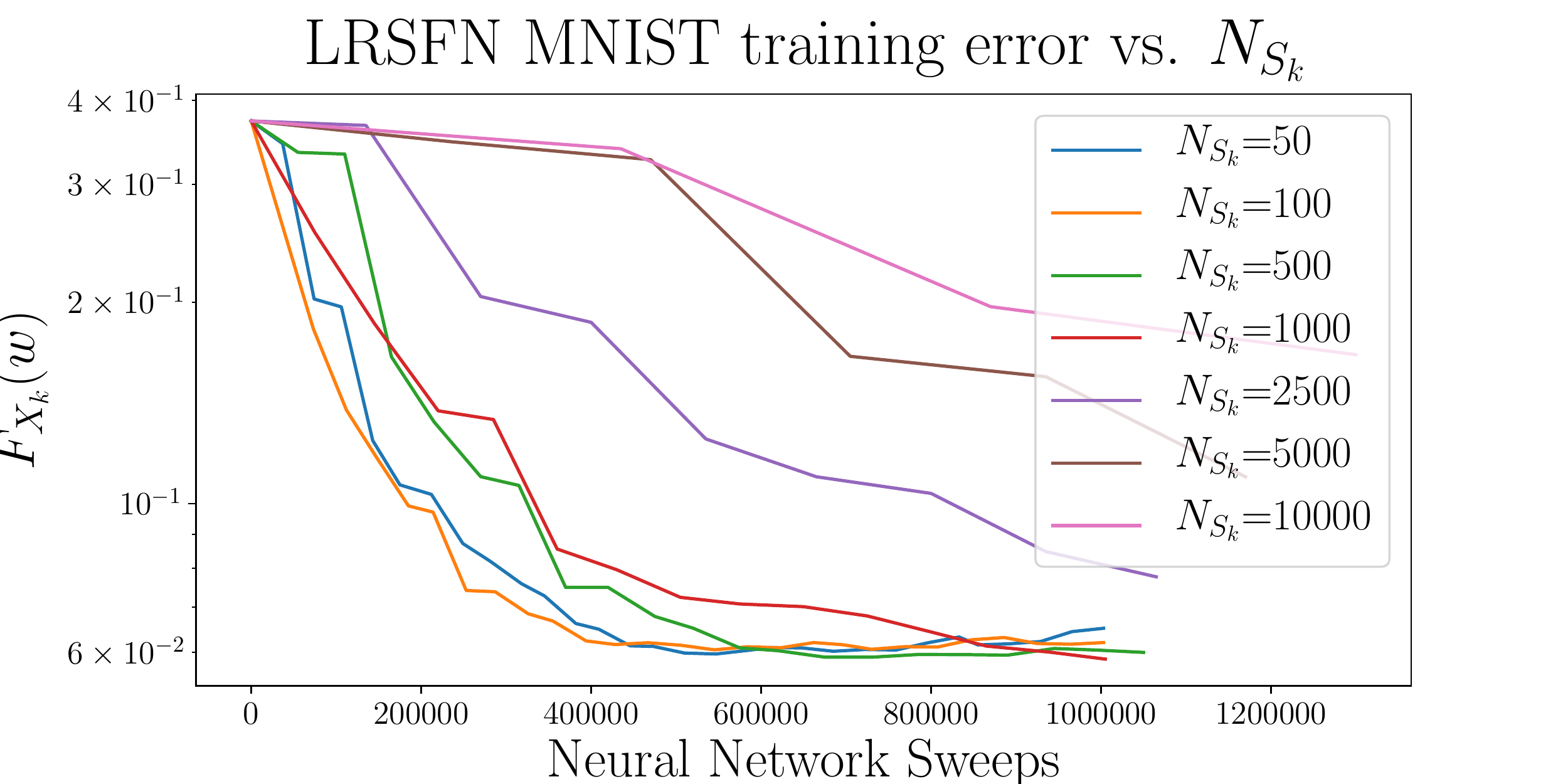}
\end{subfigure}

\caption[]{MNIST training error vs $N_{S_k}$}
\label{mnist_s_k_conv}
\end{figure}

For the MNIST ensemble runs, we take $N_{S_k} = 1000$, since we experimentally observed that $N_{S_k} = 0.1N_{X_k}$ worked well. The Hessian spectra along the path of iterates generated by INCG cluster and exhibit low rank structure throughout training, including for a random initial guess. As seen in Figure \ref{mnist_5000_eigs}, the Hessian spectra calculated using training data match Hessian spectra calculated using testing data. This suggests that the dominant modes of the Hessian did not have too much variance, and that the Hessian training batch size was sufficiently large. The numerical rank seems to be around $60$ or $70$ for this problem.

\begin{figure}[H] 
\centering
\begin{subfigure}[a]{0.5\textwidth}
   \includegraphics[width=\textwidth]{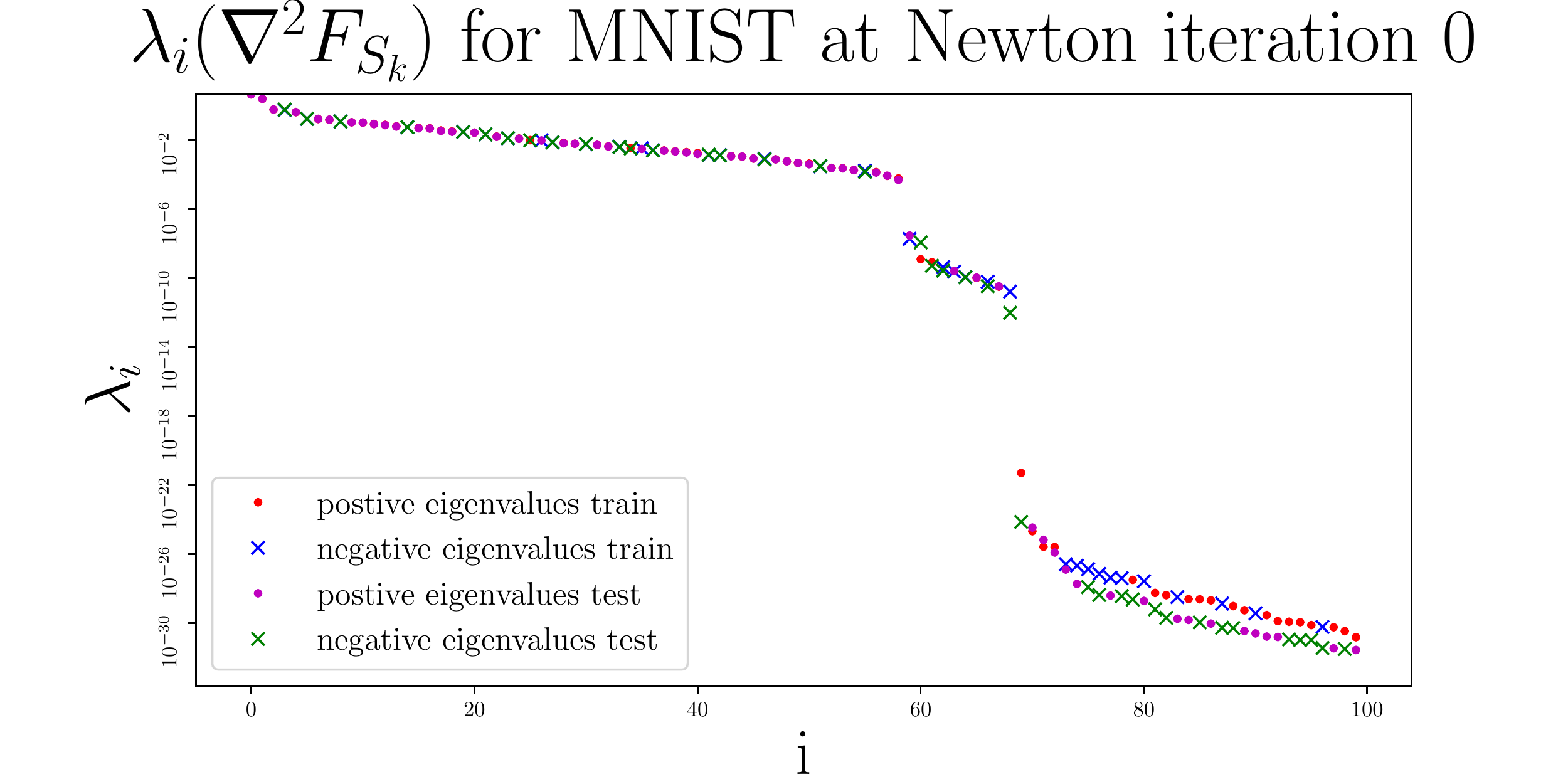}
\end{subfigure}%
\begin{subfigure}[a]{0.5\textwidth}
   \includegraphics[width=\textwidth]{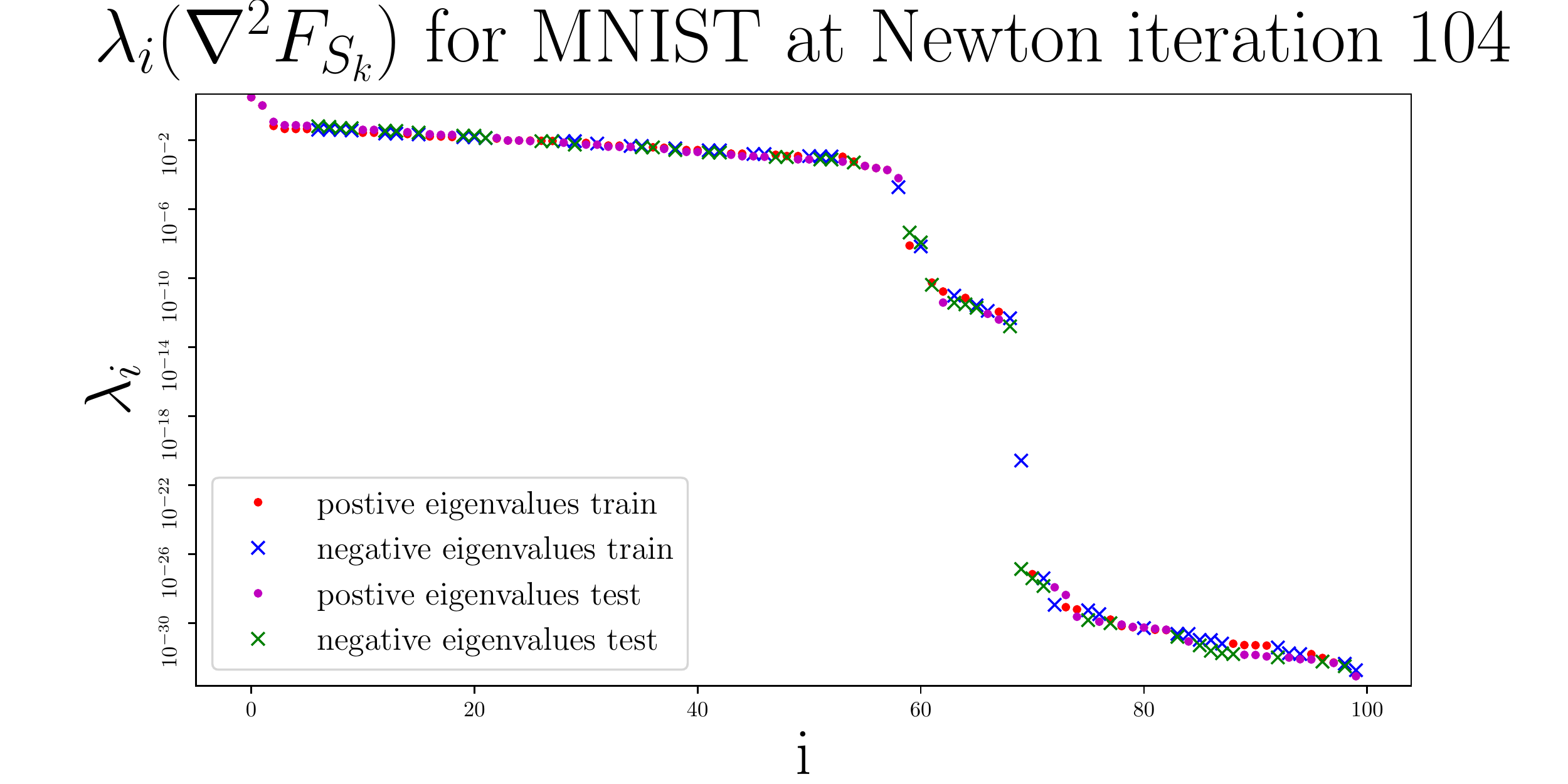}
\end{subfigure}
\caption[]{Hessian Spectra for MNIST along iterates generated by INCG}
\label{mnist_5000_eigs}
\end{figure}

In Table \ref{mnist_table}, we see that INGMRES performed the best in terms of training and testing error, followed by INCG, ADAM with $\alpha =0.01$, GD, INMINRES, LRSFN and SGD. INCG had the least variance in the testing and training error. The inexact Newton Krylov methods perfomed well overall, and LRSFN performed somewhat poorly with line search, although performed better in terms of average generalization error for the choice of $r=60$. Overall the inexact Newton Krylov methods performed well as the Hessian seemed to exhibit clustering and rapid decay throughout training.

\begin{table}[H]
\begin{center}
  \begin{tabular}{|l|c|c|c|c|}
    \hline
     \enskip &  mean$(\widehat{F}_k)$ train & std$(\widehat{F}_k)$ train & mean$(\widehat{F}_k)$ test &std$(\widehat{F}_k)$ test\\
    \hline 
    ADAM $\alpha = 0.01$ &  4.502e-02 & \textbf{6.732e-03}&  6.051e-02 & 1.294e-02 \\ \hline
    GD &  6.031e-02 & 1.333e-02&  6.108e-02 & 1.344e-02 \\ \hline
    INCG &  5.837e-02 & 1.176e-02&  5.856e-02 & \textbf{1.176e-02} \\ \hline
    INGMRES & 5.183e-02 & 1.474e-02&  \textbf{5.202e-02} & 1.477e-02 \\ \hline
    INMINRES & 7.162e-02 & 1.882e-02&  6.380e-02 & 1.635e-02 \\ \hline
    LRSFN $r=20$ &  7.057e-02 & 3.646e-02 & 7.044e-02  & 3.649e-02\\ \hline
    LRSFN $r=60$ &  6.244e-02 & 1.158e-02 & 6.227e-02  & 1.161e-02\\ \hline
    SGD $\alpha = 0.01$&  \textbf{4.453e-02} & 8.898e-03&  7.139e-02 & 1.592e-02 \\ \hline
    \hline
     \enskip &  min$(\widehat{F}_k)$ train & median$(\widehat{F}_k)$ train & min$(\widehat{F}_k)$ test &median$(\widehat{F}_k)$ test\\
    \hline 
    ADAM $\alpha = 0.01$&  2.918e-02 & 4.563e-02&  2.986e-02 & 6.024e-02 \\ \hline
    GD &  		3.048e-02 &  			6.081e-02& 		3.079e-02 & 		6.160e-02 \\ \hline
    INCG &  		3.325e-02 & 			6.074e-02&  	3.349e-02 & 		6.092e-02 \\ \hline
    INGMRES & \textbf{1.791e-02} & 4.970e-02& \textbf{1.799e-02} & \textbf{5.025e-02} \\ \hline
    INMINRES & 2.946e-02 & 7.420e-02&  2.991e-02 & 6.654e-02 \\ \hline
    LRSFN $r=20$ & 3.182e-02 & 	6.793e-02 &	3.189e-02  &  6.8434e-02\\ \hline
    LRSFN $r=60$ &  3.137e-02 & 1.161e-02 & 6.339e-02 & 6.289e-02  \\ \hline
    SGD $\alpha = 0.01$&  2.318e-02 & \textbf{4.649e-02}&  3.066e-02 & 7.207e-02 \\ \hline
  \end{tabular}
\end{center}
\caption[]{Summary for MNIST with line search over $50$ different initial guesses}
\label{mnist_table}
\end{table}


\begin{figure}[H] 
\centering
\begin{subfigure}[a]{0.5\textwidth}
   \includegraphics[width=\textwidth]{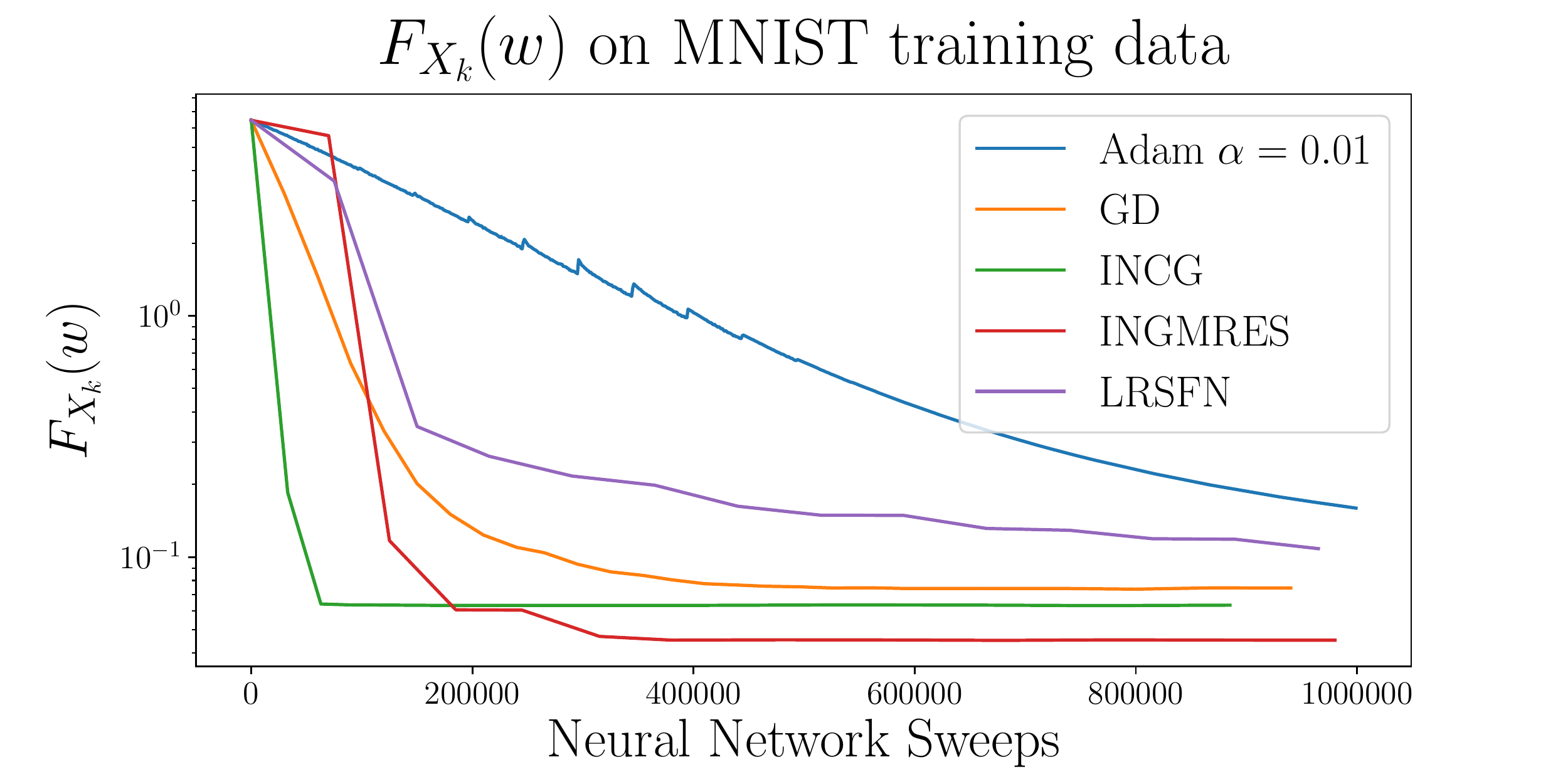}
\end{subfigure}%
\begin{subfigure}[a]{0.5\textwidth}
   \includegraphics[width=\textwidth]{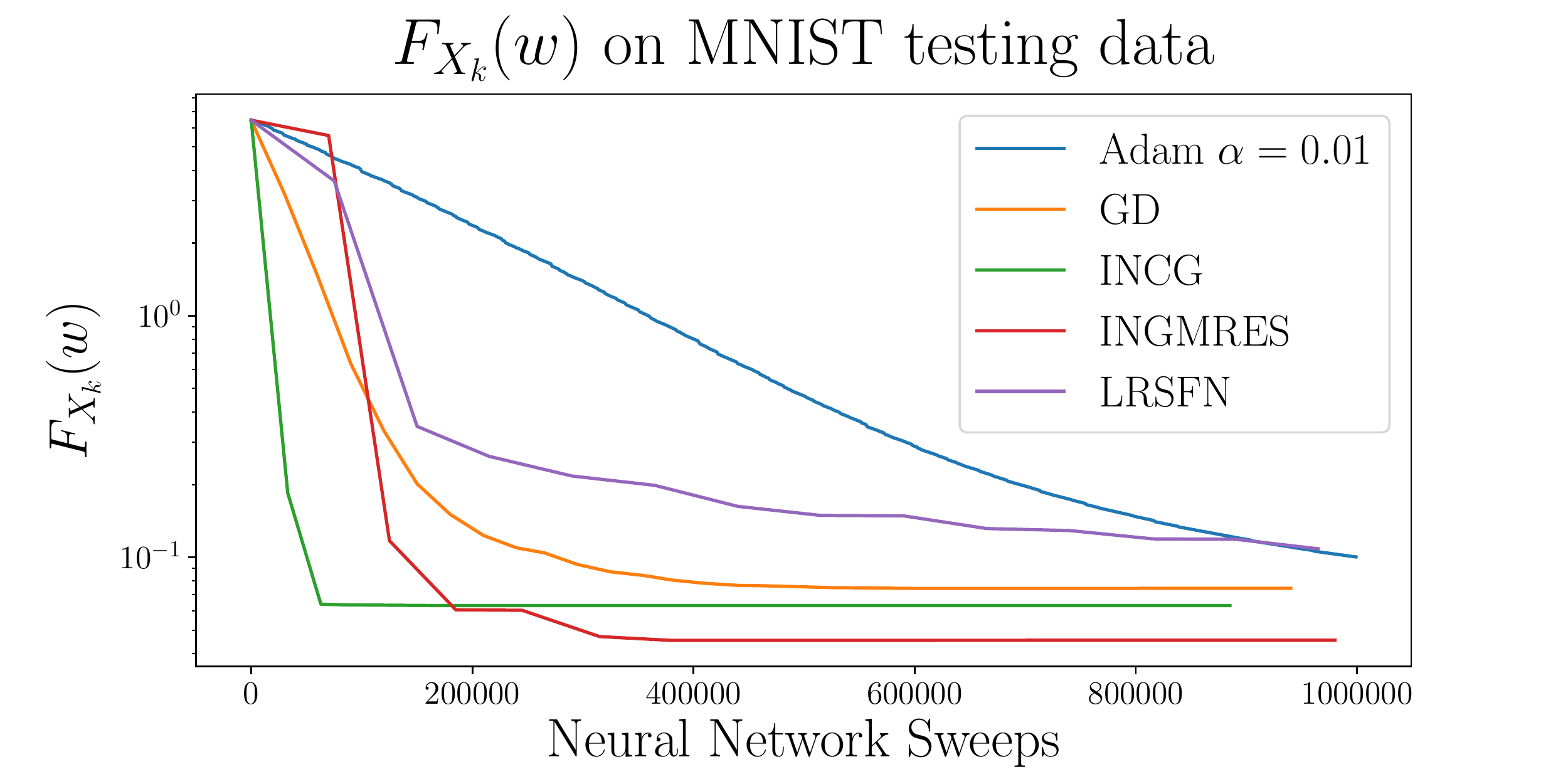}
\end{subfigure}
\caption[]{Training and testing error for a single MNIST run}
\label{mnist_10k_losses}
\end{figure}

In Figure \ref{mnist_10k_losses} we see that both INGMRES and INCG converge quickly to basins that they stay in. GD converges relatively quickly as well though to a basin worse than that of INCG and INGMRES. LRSFN converged slower for this problem, as did Adam. Both seemed to not have converged to a solution at the limit of the neural network sweeps. Note that the training and testing error plots are very similar for the line search methods, which is consistent with the results in Table \ref{mnist_table}.

\subsection{CIFAR10}

Since LRSFN did not perform well with line search for the MNIST problem (Table \ref{mnist_table}), we do not perform line search using LRSFN for CIFAR10. We begin again by studying the effect of the Hessian batch size $N_{S_k}$ on convergence. As Figure \ref{cifar_s_k_conv_incg} shows, the INCG runs all performed about the same in terms of Newton iterations. The methods with smaller batch sizes ended up overfitting, since the small batch Newton iterations are so cheap, they perform many more of them. The same pattern is evident in Figure \ref{cifar_s_k_conv_lrsfn} for the LRSFN iterates. Since the large Hessian batch size runs did not perform better than the small batch size runs this suggests there is redundancy in the Hessian data. Taking $N_{S_k}$ small is again justified, but terminating the methods before they begin to overfit is also warranted.

\begin{figure}[H] 
\centering
\begin{subfigure}[a]{0.5\textwidth}
   \includegraphics[width=\textwidth]{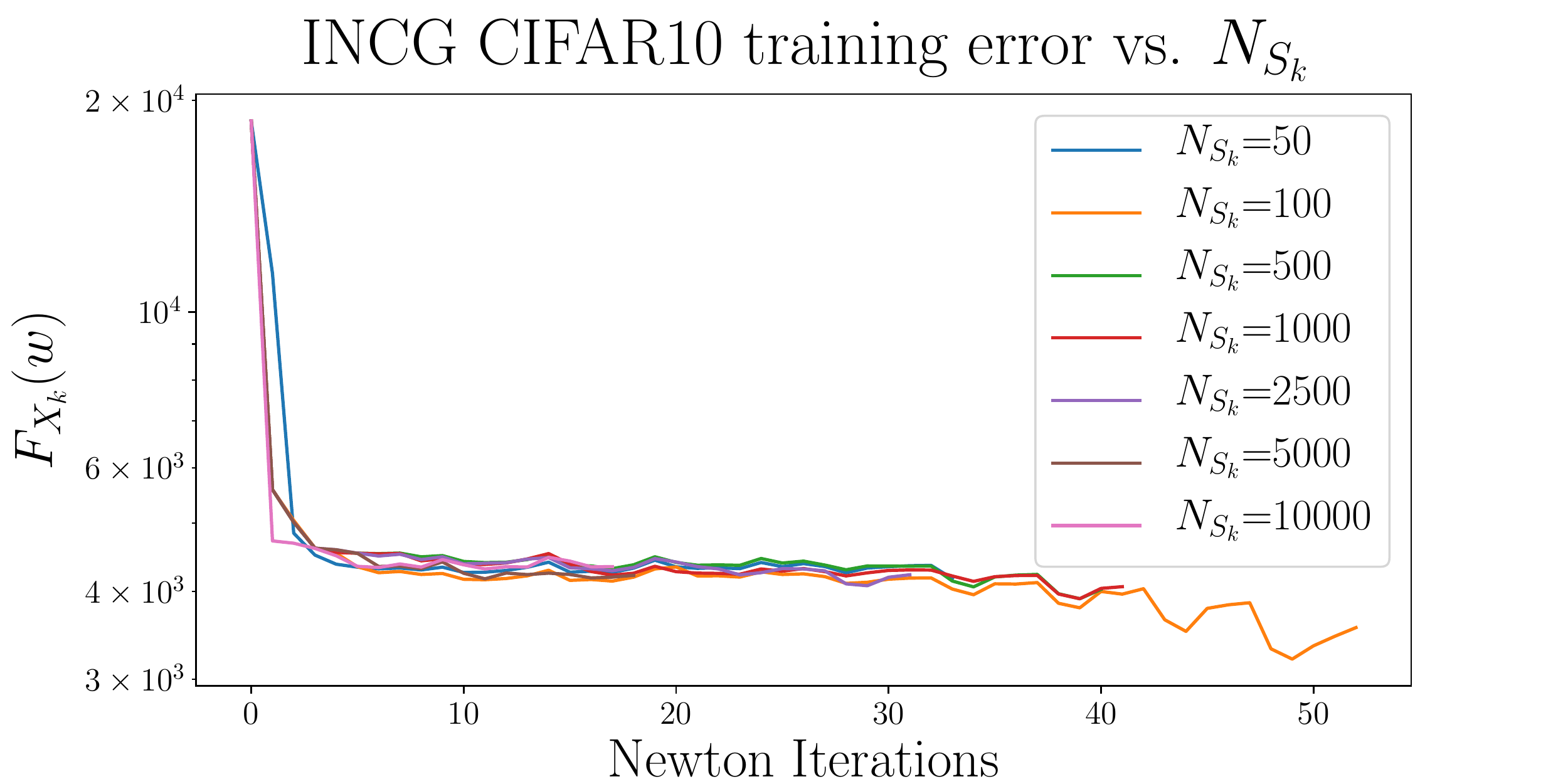}
\end{subfigure}%
\begin{subfigure}[a]{0.5\textwidth}
   \includegraphics[width=\textwidth]{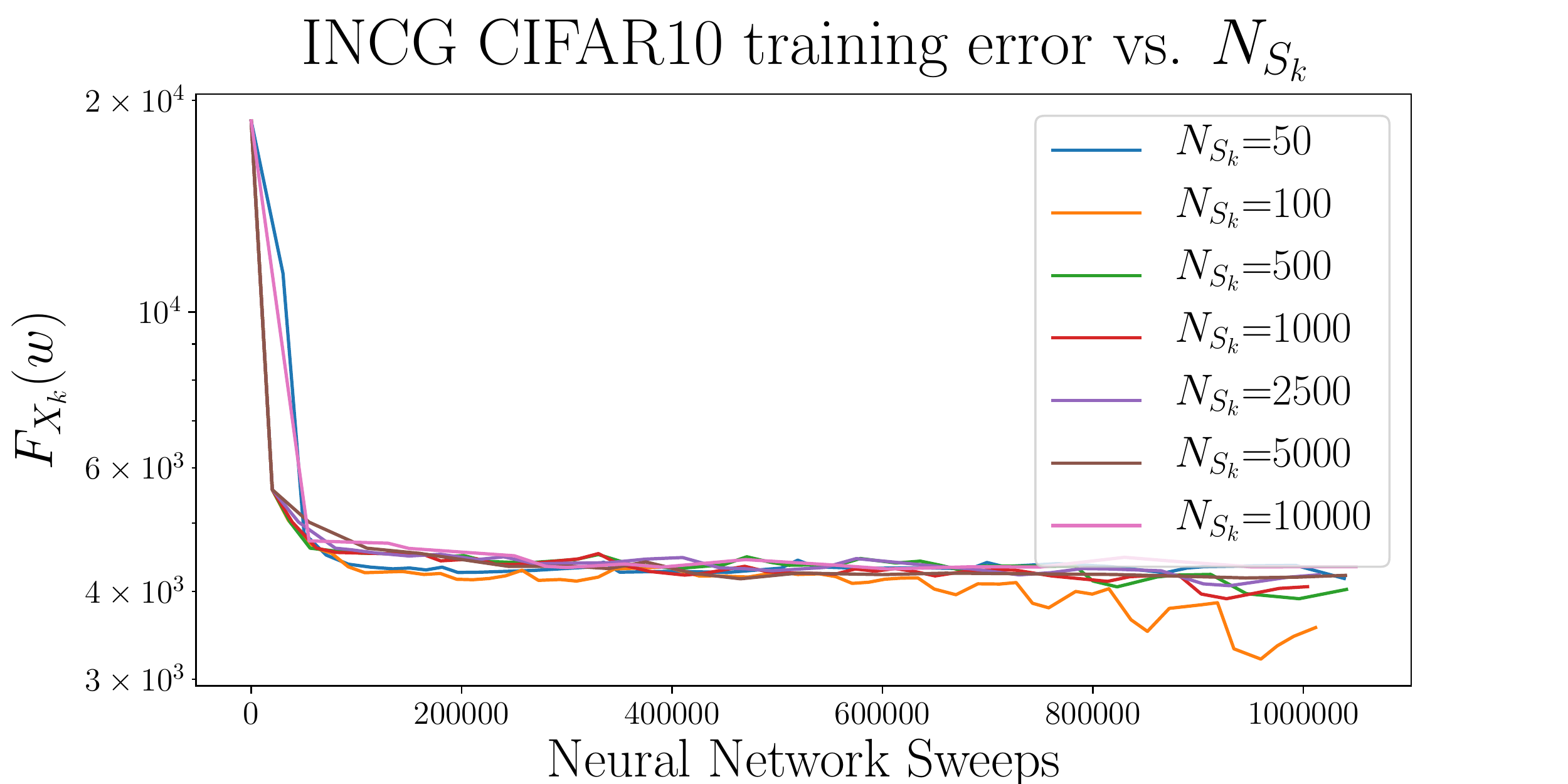}
\end{subfigure}
\begin{subfigure}[b]{0.5\textwidth}
   \includegraphics[width=\textwidth]{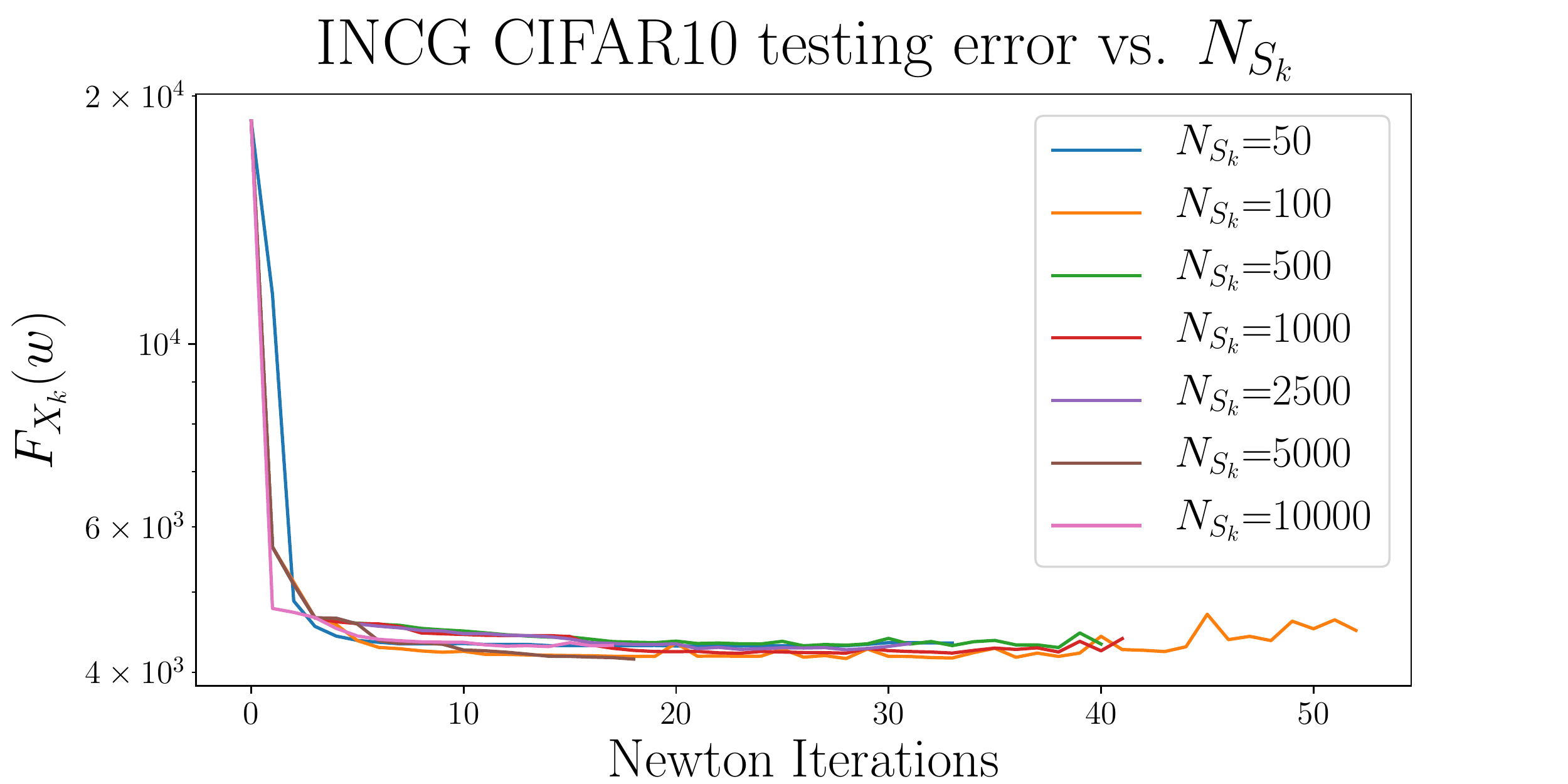}
\end{subfigure}%
\begin{subfigure}[b]{0.5\textwidth}
   \includegraphics[width=\textwidth]{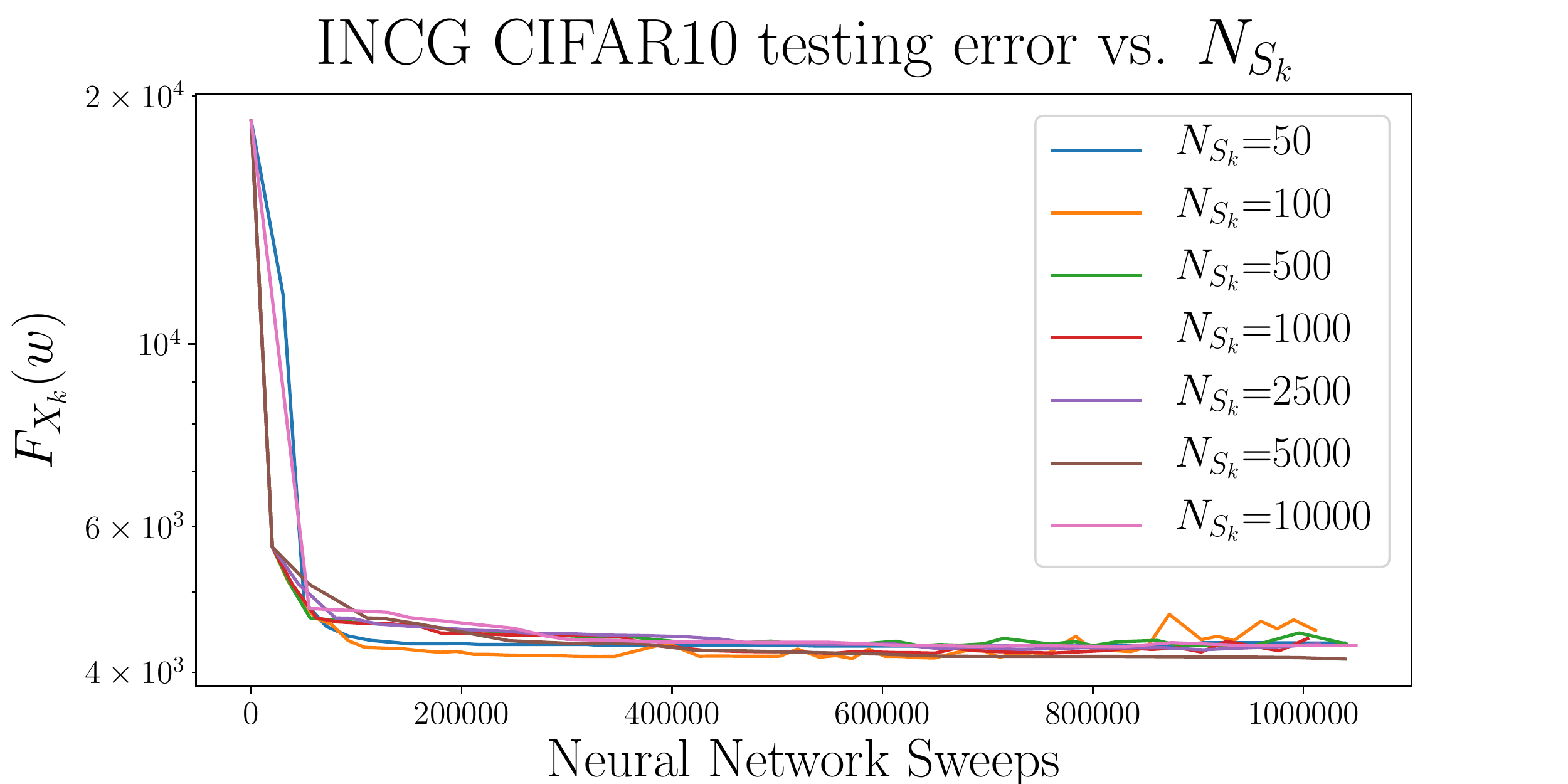}
\end{subfigure}

\caption[]{Top: CIFAR10 training error vs $N_{S_k}$; Bottom: CIFAR10 testing error vs $N_{S_k}$ }
\label{cifar_s_k_conv_incg}
\end{figure}

\begin{figure}[H] 
\centering
\begin{subfigure}[b]{0.5\textwidth}
   \includegraphics[width=\textwidth]{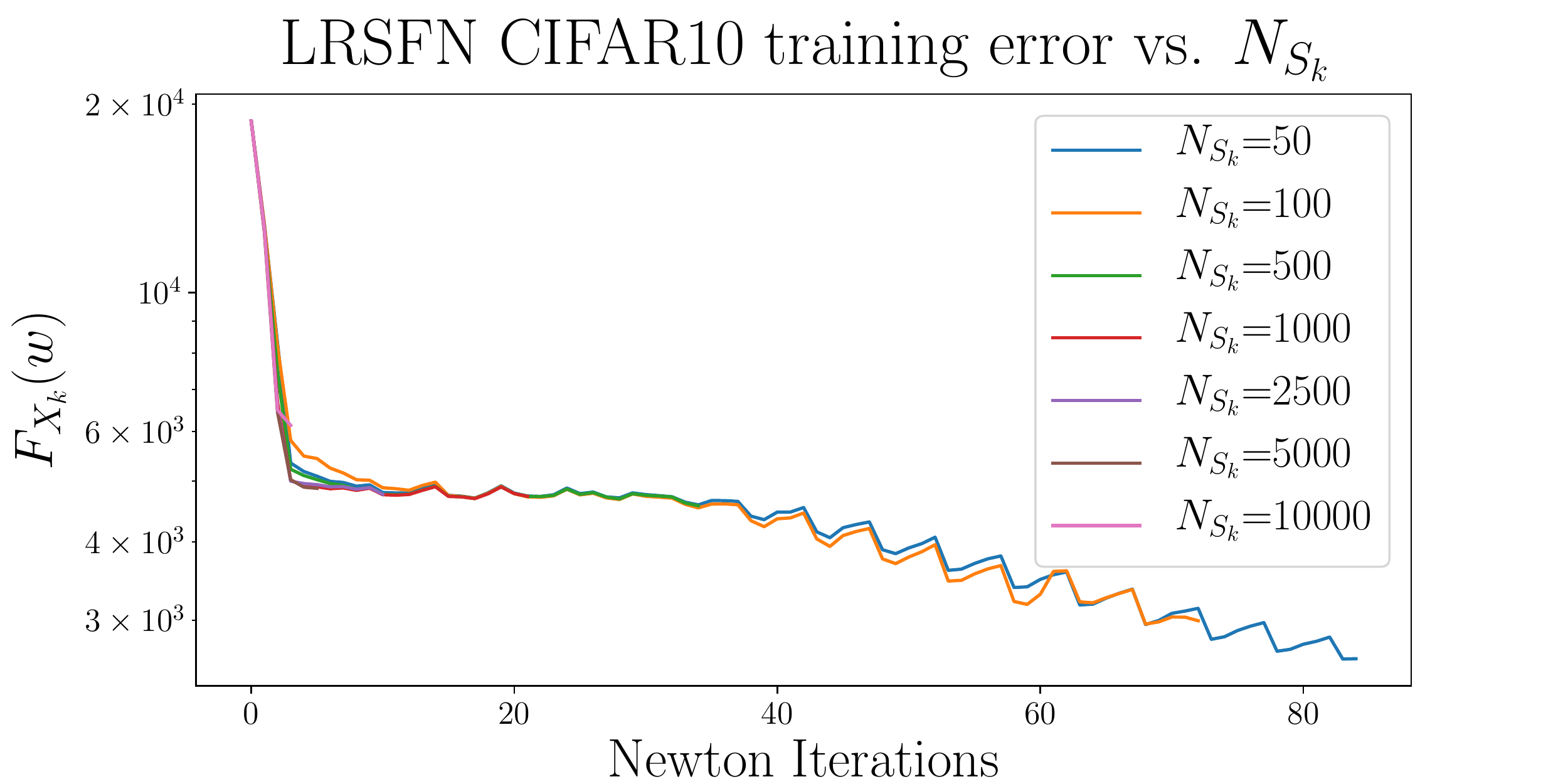}
\end{subfigure}%
\begin{subfigure}[b]{0.5\textwidth}
   \includegraphics[width=\textwidth]{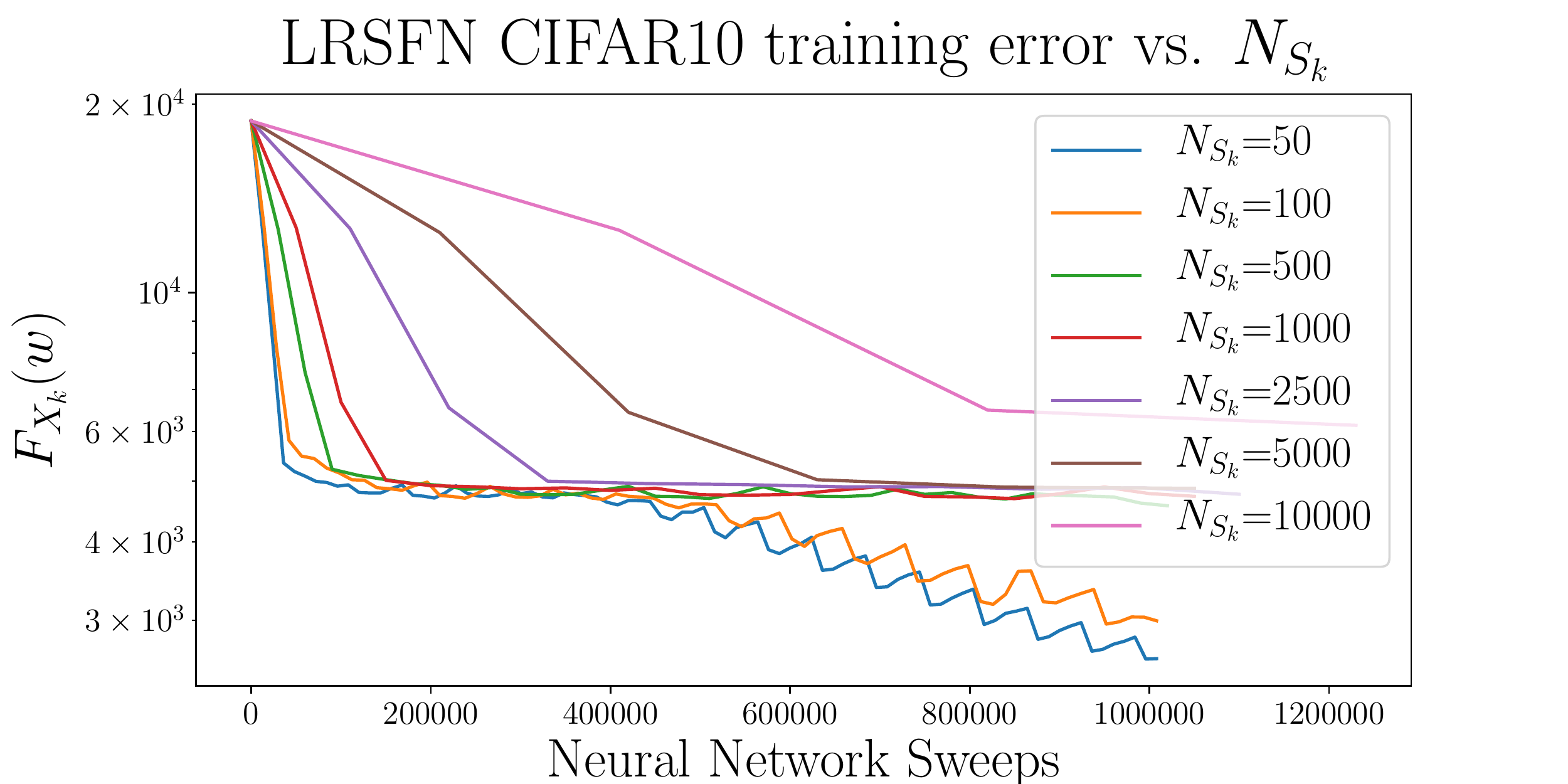}
\end{subfigure}
\begin{subfigure}[b]{0.5\textwidth}
   \includegraphics[width=\textwidth]{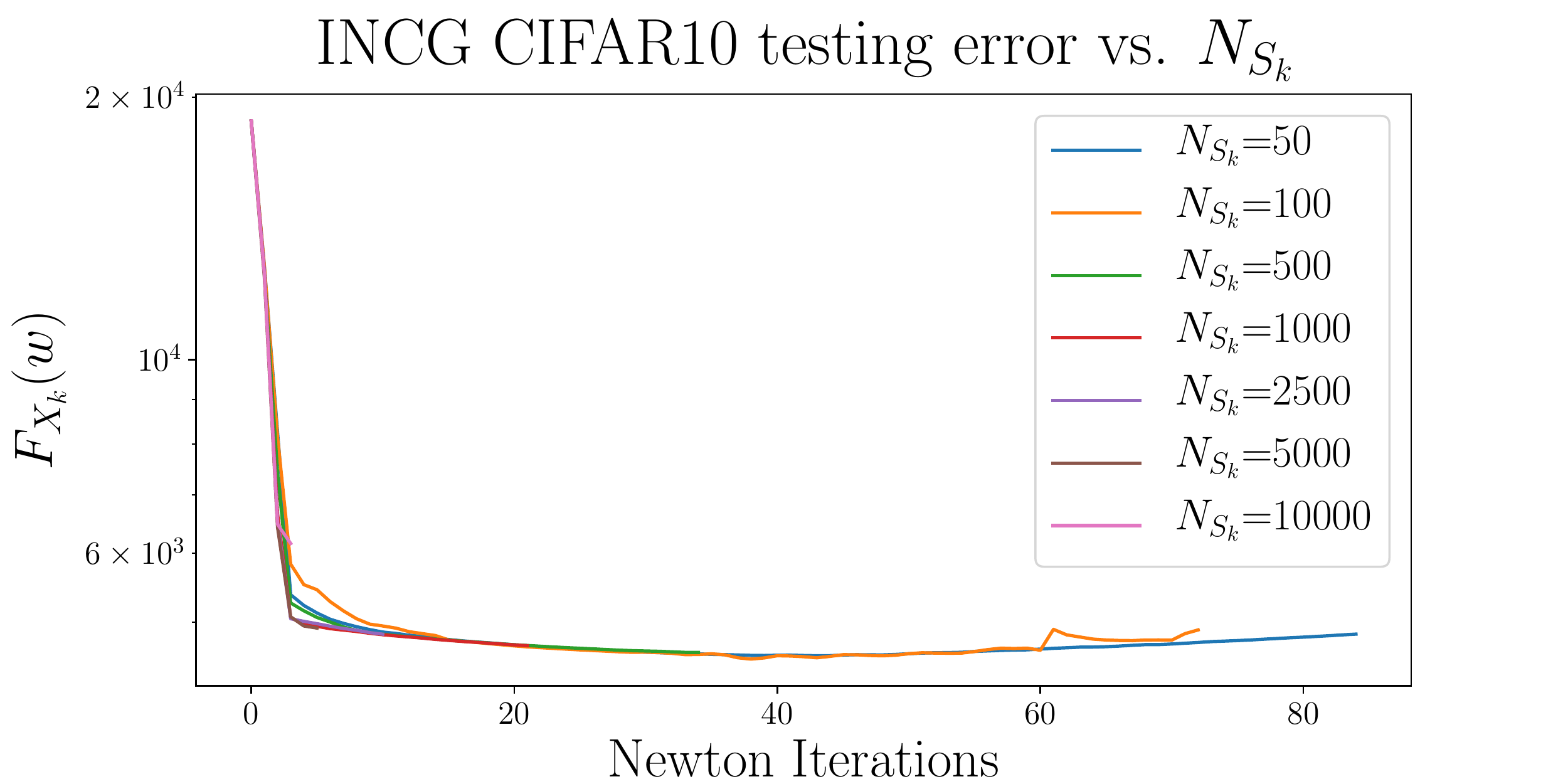}
\end{subfigure}%
\begin{subfigure}[b]{0.5\textwidth}
   \includegraphics[width=\textwidth]{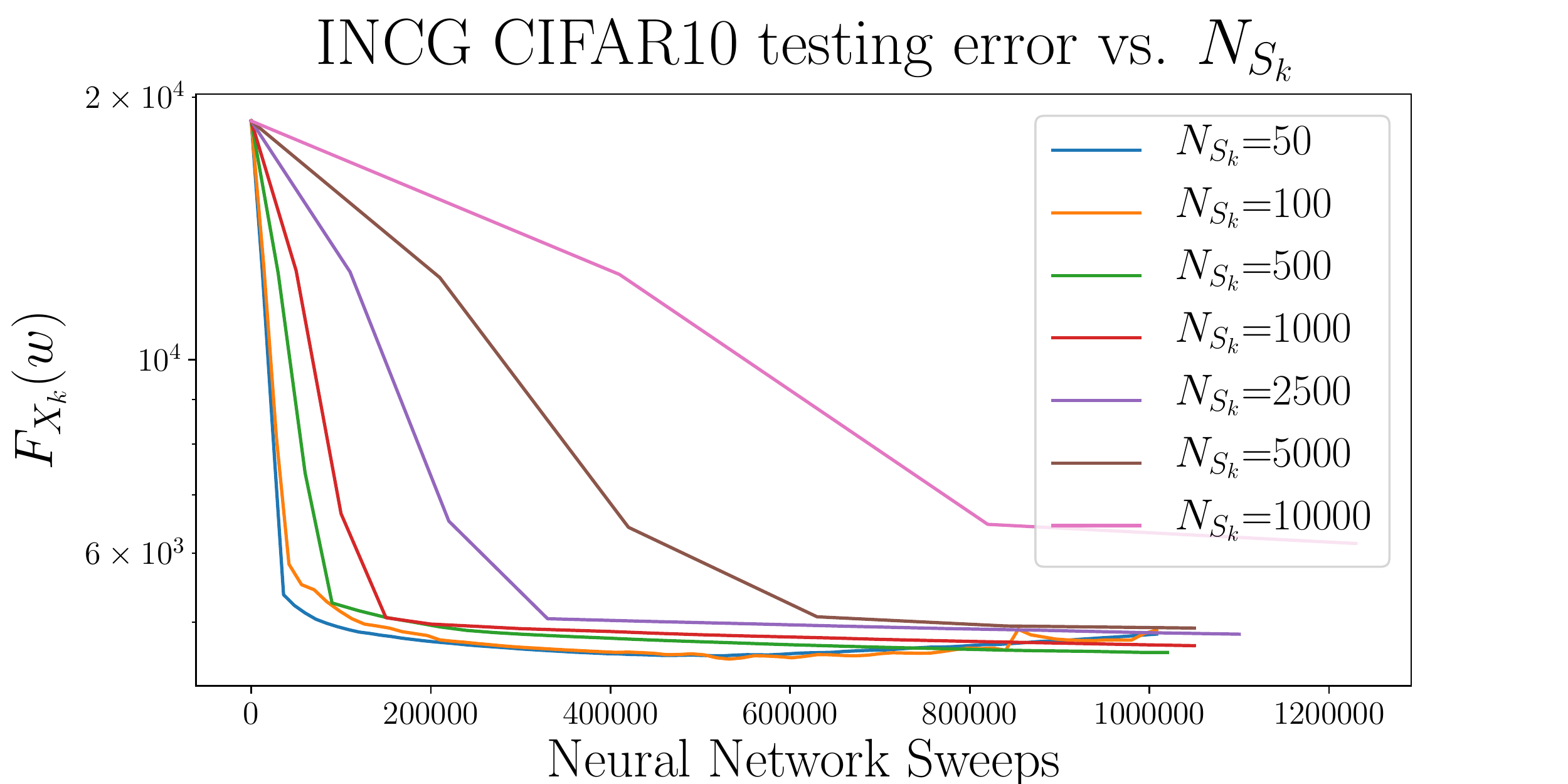}
\end{subfigure}

\caption[]{Top: CIFAR10 training error vs $N_{S_k}$; Bottom: CIFAR10 testing error vs $N_{S_k}$}
\label{cifar_s_k_conv_lrsfn}
\end{figure}

Figure \ref{cifar_10k_eigs} shows that for a random initial guess the spectrum was slow to decay and did not show any clustering properties in the top 100 eigenvalues. However shortly after taking a few steps the spectrum decayed much faster and started exhibiting clusters. Note also that these plots suggest that while the most negative eigenvalue may be large for random initial guesses and early iterations, eventually the largest (in magnitude) negative eigenvalue may become small asymptotically (as the iterates near an optimum), which justifies the assumption that $-\epsilon_H I \preceq \nabla^2 F_{S_k}$ for $\epsilon_H$ small used in the derivation of local convergence rates. At iteration $20$ the largest negative eigenvalue for the training data Hessian is approximately $10^4$, whereas at iteration $30$ the largest negative eigenvalue for the training data Hessian is approximately $10^{-2}$ There is more variance between the testing and training spectra in this problem than in the MNIST problem.

\begin{figure}[H] 
\centering
\begin{subfigure}[a]{0.5\textwidth}
   \includegraphics[width=\textwidth]{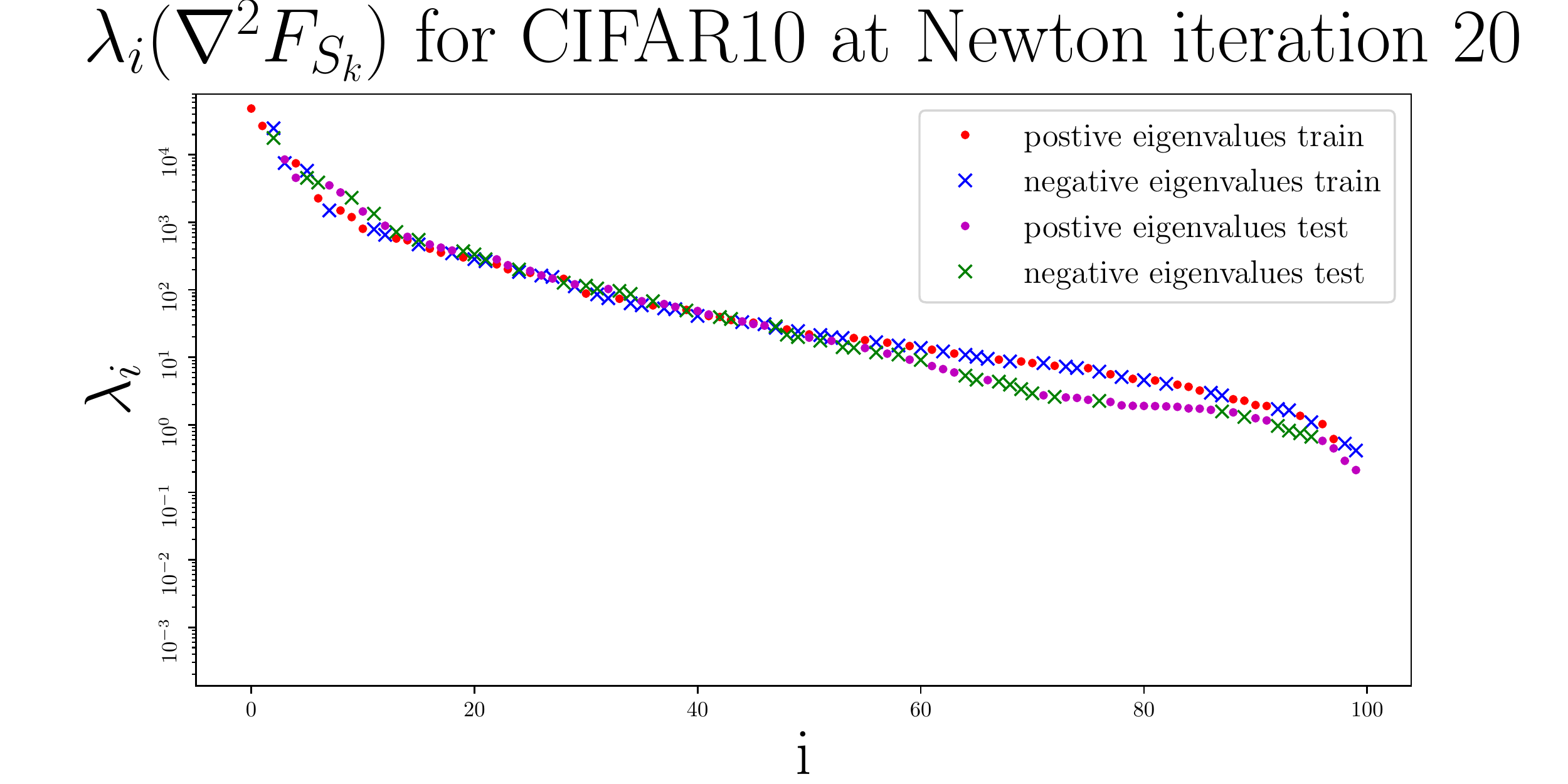}
\end{subfigure}\hfill%
\begin{subfigure}[a]{0.5\textwidth}
   \includegraphics[width=\textwidth]{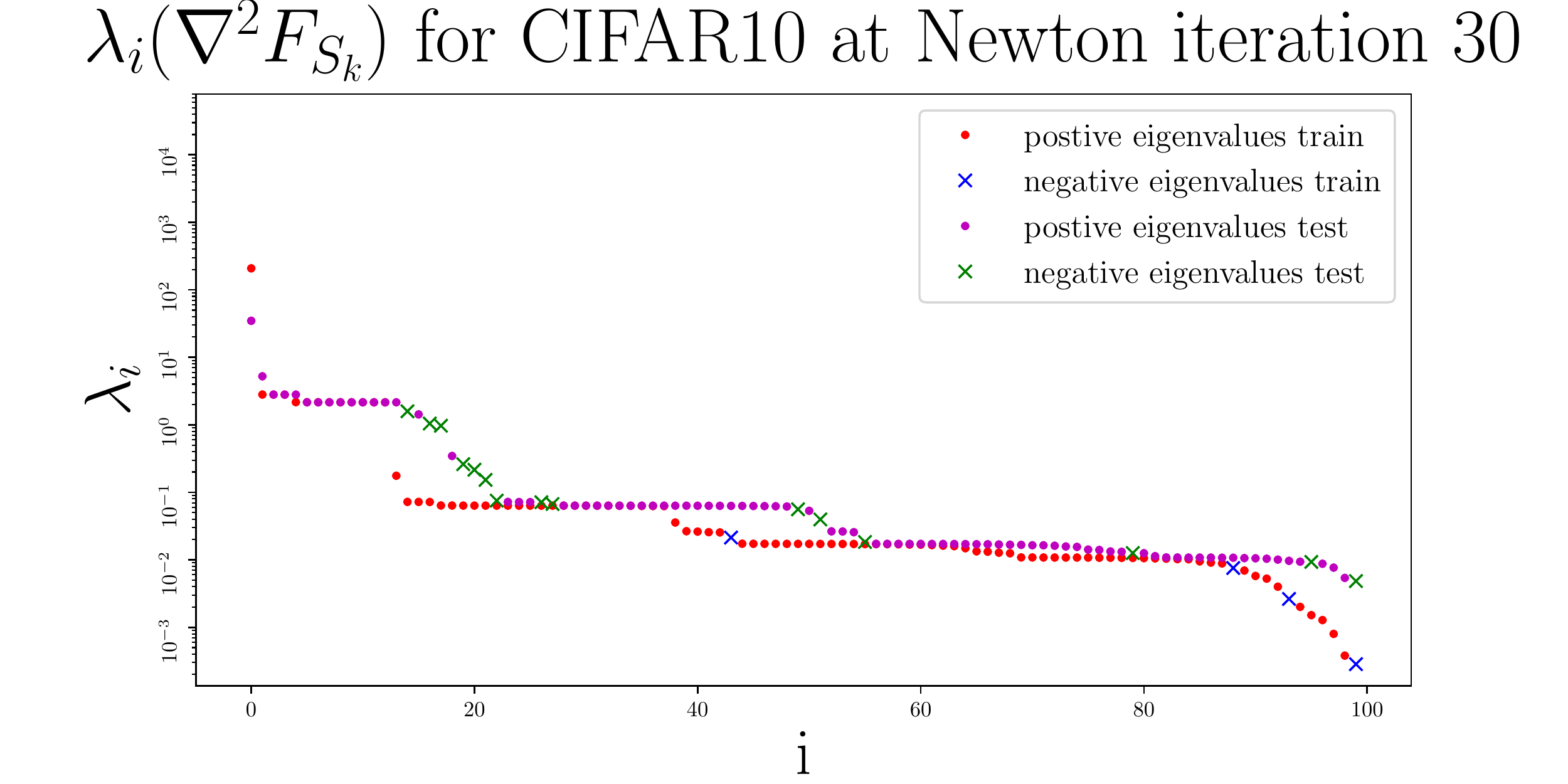}
\end{subfigure}
\caption[]{Hessian Spectra for CIFAR10 along iterates generated by INCG}
\label{cifar_10k_eigs}
\end{figure}

Since we take fixed step lengths for LRSFN for the CIFAR10 problem, we split the numerical tests into two categories: line search methods and fixed step methods. The inexact Newton GMRES and MINRES methods did not perform well for random initial guesses. We believe this is due to the high rank structure of the Hessian spectrum for random initial guesses seen in Figure \ref{cifar_10k_eigs}. In some problems such as shape optimization it can be shown that the Gauss-Newton Hessian is a compact operator \cite{BuiGhattas2012}. In such settings where high rank Hessians are observed and lead to poor Newton performance, the Gauss-Newton Hessian can be substituted for a few iterations until the full Hessian exhibits low rank structure again. In such cases the second order terms not in the Gauss-Newton Hessian dominate. With this in mind, for INMINRES and INGMRES we perform two iterations of gradient descent (instead of the Gauss-Newton Hessian). After these iterations the spectrum exhibits low rank clustering structure and the MINRES and GMRES methods perform very well. We denote this with $\dagger$.Table \ref{cifar_table_line_search} shows that the inexact Newton-Krylov methods all outperformed gradient descent with line search. The fully stochastic INCG algorithm performed the best in generalization error.

\begin{table}[H]
\begin{center}
  \begin{tabular}{|l|c|c|c|c|}
    \hline
      \multicolumn{5}{|c|}{ $N_{X_k} = 10000,N_{S_k} = 1000, d = 12315$, $|x| = 3(32)^2$   }  \\\hline \hline
     \enskip &  mean$(\widehat{F}_k)$ train & std$(\widehat{F}_k)$ train & mean$(\widehat{F}_k)$ test &std$(\widehat{F}_k)$ test\\
    \hline 
    GD &  4.345e+03 & 109.2 &  4.397e+03 & 106.0 \\ \hline
    INCG &  4.230e+03 & \textbf{62.34} &  4.285e+03 & \textbf{59.77} \\ \hline
    INCG SA &  \textbf{3.872e+03} & 204.9 &  \textbf{4.144e+03} & 109.3  \\ \hline
    $\text{INGMRES}^\dagger$ & 4.305e+03 & 94.77&  4.338e+03 & 93.28 \\ \hline
    $\text{INMINRES}^\dagger$ & 4.272e+03 & 151.9&  4.350e+03 & 63.96 \\ \hline
  \end{tabular}
\end{center}
\caption[]{\textbf{(Line search)} Summary for CIFAR10 over 50 different initial guesses}
\label{cifar_table_line_search}
\end{table}

As Table \ref{cifar_table_fixed_step} shows, the fully stochastic version of LRSFN performed the best, and all of the methods performed better for $\alpha = 0.05$ than $\alpha = 0.01$. The methods were numerically unstable for $\alpha = 0.1$, and those results were omitted. SA LRSFN with fixed step and SA INCG with line search performed the best for the CIFAR10 problem overall.

\begin{table}[H]
\begin{center}
  \begin{tabular}{|l|c|c|c|c|}
    \hline
      \multicolumn{5}{|c|}{ $N_{X_k} = 10000,N_{S_k} = 500, d = 12315$, $|x| = 3(32)^2$   }  \\\hline \hline
     \enskip &  mean$(\widehat{F}_k)$ train & std$(\widehat{F}_k)$ train & mean$(\widehat{F}_k)$ test &std$(\widehat{F}_k)$ test\\
    \hline 
    Adam $\alpha = 0.01$ &  1.350e+03 & 176.6 &  5.204e+03 & 24.12 \\ \hline
    Adam $\alpha = 0.05$ &  9.593e+02 & 79.62 &  4.804e+03 & 139.0 \\ \hline
    LRSFN $\alpha = 0.01$ &  3.290e+03 & 34.97 &  4.539e+03 & 49.86 \\ \hline
    LRSFN $\alpha = 0.05$ &  3.128e+03 & 116.5 &  4.307e+03 & \textbf{22.36} \\ \hline
    LRSFN SA $\alpha = 0.01$ &  4.126e+03 & 108.1 &  \textbf{4.264e+03} & 135.3 \\ \hline
    LRSFN SA $\alpha = 0.05$ &  4.113e+03 & \textbf{23.24} &  4.268e+03 & 135.3 \\ \hline
    SGD $\alpha = 0.01$ & \textbf{8.041e+02} & 161.2 & 4.544e+03 & 116.8 \\ \hline
    SGD $\alpha = 0.05$& 8.140e+02 & 117.2 & 4.333e+03 & 59.9 \\ \hline
  \end{tabular}
\end{center}
\caption[]{(\textbf{Fixed step length}) Summary for CIFAR10 over 50 different initial guesses ($\alpha_k = 0.01$)}
\label{cifar_table_fixed_step}
\end{table}

\begin{figure}[H] 
\centering
\begin{subfigure}[a]{0.5\textwidth}
   \includegraphics[width=\textwidth]{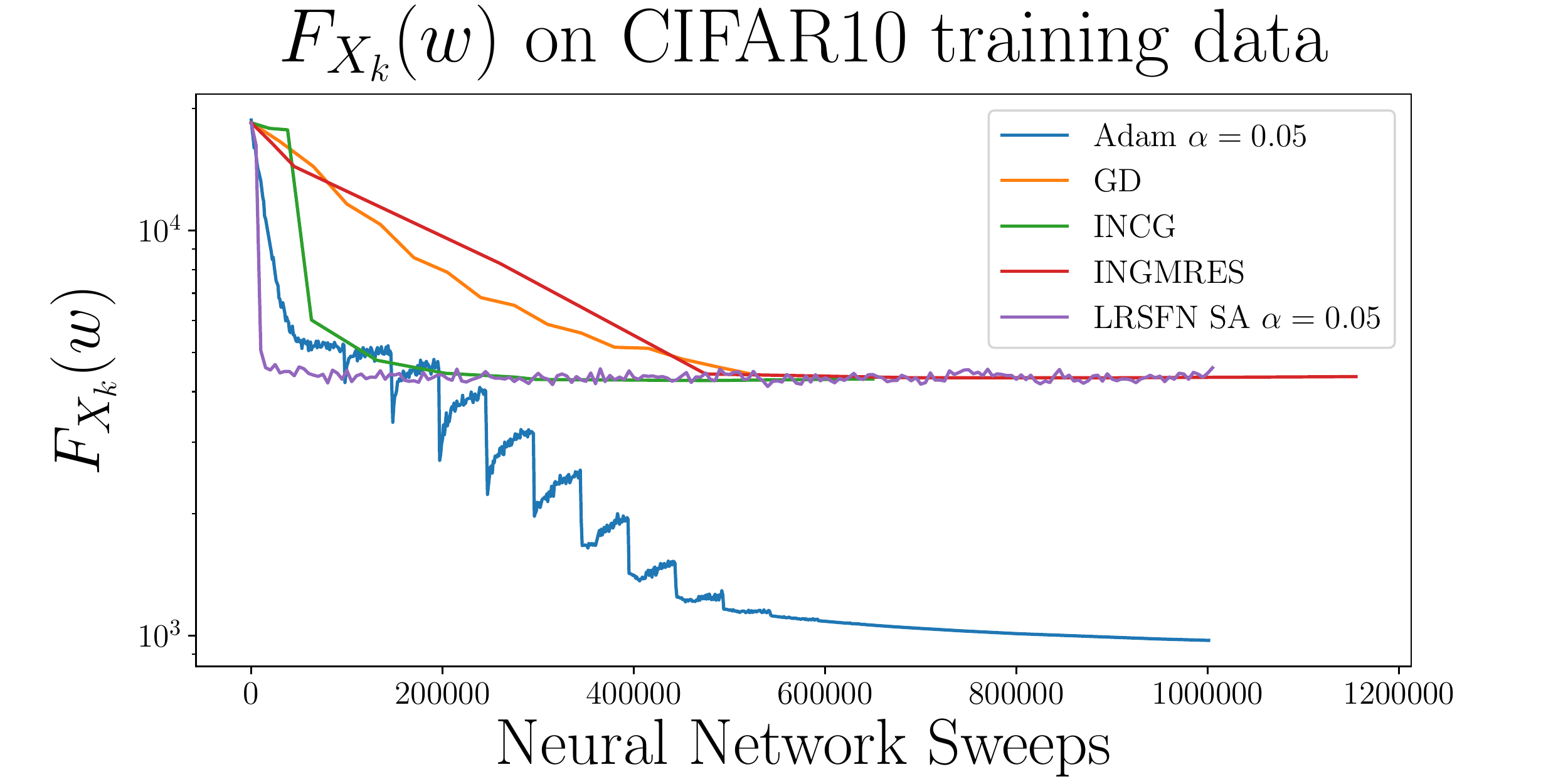}
\end{subfigure}%
\begin{subfigure}[a]{0.5\textwidth}
   \includegraphics[width=\textwidth]{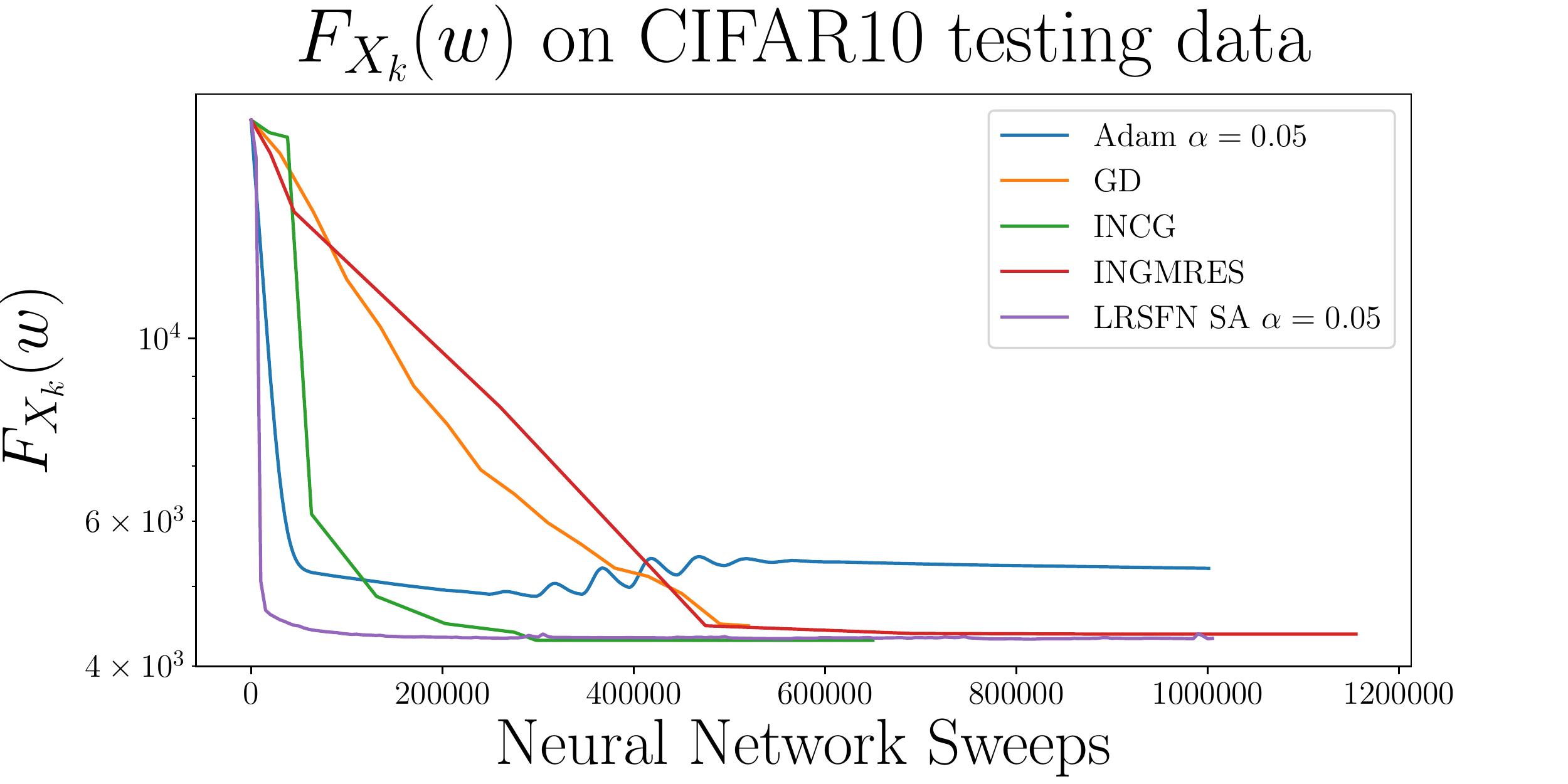}
\end{subfigure}
\caption[]{Training and testing error for a single CIFAR10 run}
\label{cifar10_10k_losses}
\end{figure}

In Figure \ref{cifar10_10k_losses} we see that both of the SA algorithms (Adam and LRSFN) performed the best in terms of training error. Adam however overfits drastically. INCG, INGMRES and GD all converge to basins comparable to that of LRSFN, INCG converges relatively fast, while GD and INGMRES are slower.

\subsection{Software}

All of the code used for this paper can be found at \url{https://github.com/tomoleary/hessianlearn}, a Python library for second order stochastic optimization in TensorFlow \cite{AbadiBarhamChenEtAl2016}. All results were generated on servers with 64 Intel Xeon Skylake cores, and 1 TB of RAM.

\section{Conclusion}

In this paper we have studied stochastic inexact Newton methods and their application to nonconvex optimization problems. We presented analysis on stochastic low rank Newton methods, and stochastic inexact Newton Krylov methods. We derived convergence rates that characterize the errors incurred in subsampling the gradient and Hessian, in approximately solving the Newton linear system, and in choosing step length and regularization hyperparameters.

We performed numerical experiments on MNIST and CIFAR10 convolutional autoencoder training problems. Specifically we studied both fully-stochastic and semi-stochastic variants of these Newton methods in comparison with standard methods such as Adam, gradient descent, and stochastic gradient descent. We studied the effect of Hessian batch size on training and generalization. We saw that if too little Hessian information was used the methods may be susceptible performing too many cheap iterations and overfitting, while too much Hessian information made the problem harder to solve.

In these experiments we found that inexact Newton Krylov methods performed the best. LRSFN did not perform well when line search was enforced. However for fixed step lengths fully stochastic and semi-stochastic variants of LRSFN performed better. In every case at least one inexact Newton method outperformed the best first order method in generalization error. Theory suggests that the Krylov methods would better approximate the Newton linear system for a given number of Hessian vector products, when clustering is prevalent in the Hessian spectrum. In the case of neural network training, we observed clustering of the Hessian spectrum in our numerical experiments, in agreement with other numerical experiments \cite{AlainRouxManzagol2019,GhorbanKrishnanXiaoi2019,SagunBottouLeCun2016}. We believe that the Krylov methods performed the best due to the clustering properties of the Hessian spectra. 

The inexact Newton Krylov methods break down when the Hessian is high rank or does not have clustered eigenvalues. In the highly data informed regime, where the Hessian likely has high rank, preconditioning may be critical to the success of these methods. LRSFN does not perform well when the Hessian spectrum decays very slowly, and cannot take advantage of the clustering properties that the Krylov methods do. A saddle free Newton algorithm that explicitly uses a Krylov linear solver (not just Lanczos vectors) may perform better in this case.


\newpage

\bibliographystyle{siam}
\bibliography{local}

\newpage

\appendix

\section{Local Convergence Rate Estimates} \label{local_convergence}


\begin{lemma}\label{lipschitz_bound_alpha_k}
Suppose that assumptions \ref{assumption_a1} and \ref{assumption_a3} hold and $\alpha_k > 0$. Then the following bound holds:
\begin{equation}\label{lipschitz_bound_alpha_k_eq}
	\|\nabla^2 F(w_k)(w_k - w^*) - \alpha_k \nabla F (w_k)\| \leq \frac{M}{2}\|w_k - w^*\|^2 + L_{N_{S_k}}|1 - \alpha_k| \|w_k - w^*\|.
\end{equation}
\end{lemma}

\begin{proof}
The triangle inequality allows us to split the left hand side of \eqref{lipschitz_bound_alpha_k_eq} as follows:
\begin{align}
	&\|\nabla^2 F(w_k)(w_k - w^*) - \alpha_k \nabla F (w_k)\| \nonumber \\
  &= \|\nabla^2 F(w_k)(w_k - w^*) - \nabla F (w_k) + (1-\alpha_k)\nabla F(w_k)\| \nonumber \\
	&\leq \underbrace{\|\nabla^2 F(w_k)(w_k - w^*) - \nabla F (w_k)\|}_{\text{term 1}} + |1-\alpha_k|\underbrace{\|\nabla F(w_k)\|}_{\text{term 2}}.
\end{align}
By a derivation in Lemma 2.2 in \cite{BollapragadaByrdNocedal2018} that uses the Lipschitz continuity of the Hessian, we can bound term 1 by
\begin{equation}
	\|\nabla^2 F(w_k)(w_k - w^*) - \nabla F (w_k)\| \leq \frac{M}{2}\|w_k - w^*\|^2.
\end{equation}
Defining $h(t) = \nabla F(w^* + t(w_k - w^*))$, we may bound term 2 as follows:

\begin{align}
\|\nabla F(w_k)\| & = \big|\|\nabla F(w_k)\| - \|\nabla F(w^*)\| \big| \nonumber \\
                  & \leq \| \nabla F(w_k) - \nabla F(w^*)\| \nonumber \\
                  & = \| h(1) - h(0)\| \nonumber \\
                  & = \| \int_0^1 h'(t)dt \| \nonumber \\
                  & \leq  \int_0^1 \| h'(t)\|dt \nonumber \\
                  & = \int_0^1 \|\nabla^2 F(w^* + t(w_k - w*))(w_k - w^*) \| dt  \nonumber \\
                  & \leq \int_0^1 L_{N_{S_k}} \|w_k - w^*\|dt \nonumber \\
                  & = L_{N_{S_k}} \| w_k - w^*\|.
\end{align}
\end{proof}

\begin{lemma}\label{stochastic_newton_bound_alpha_k}
Suppose that assumptions \ref{assumption_a1}-\ref{assumption_a4} hold, and $\alpha_k > 0$, then
\begin{align}
  &\mathbb{E}_k[\| \nabla^2 F_{S_k}(w_k)(w_k -w^*)  - \alpha_k\nabla F_{X_k}(w_k) \|] \leq \nonumber \\ 
  &\frac{M}{2}\|w_k - w^*\|^2 + \bigg( L_{N_{S_k}}|1 - \alpha_k|  +\frac{\sigma}{\sqrt{N_{S_k}}}\bigg)\|w_k - w^*\| + \frac{\alpha_k v}{\sqrt{N_{X_k}}}.
\end{align}
\end{lemma}
\begin{proof}
This result follows immediately from Lemma \ref{lipschitz_bound_alpha_k} and Lemmas 2.2 and 2.3 in \cite{BollapragadaByrdNocedal2018}.
\end{proof}

\subsection{Proof of Theorem \ref{inexact_eisenprop}}\label{proof_of_inexact_eisenprop}
\enskip \\
\inexacteisen*
\begin{proof}
We begin by expanding the left hand side of equation \ref{inexact_eisenprop_bound} and employing the triangle inequality:
\begin{align}
  &\mathbb{E}_k[\|w_{k+1} - w^*\|] \nonumber\\
  & =\mathbb{E}_k[ \|w_k - w^* - \alpha_k p_k\|] \nonumber\\
             &= \mathbb{E}_k[\|\nabla^2 \overline{F}_{S_k}(w_k)^{-1}(\nabla^2\overline{F}_{S_k}(w_k)(w_k - w^*) \nonumber \\
              & \quad - \alpha_k\nabla^2\overline{F}_{S_k}(w_k)p_k + \alpha_k \nabla \overline{F}_{X_k}(w_k) - \alpha_k\nabla \overline{F}_{X_k}(w_k))\|] \nonumber \\
             &\leq \frac{1}{\gamma - \epsilon_H}\underbrace{\mathbb{E}_k[\| \nabla^2 \overline{F}_{S_k}(w_k)(w_k -w^*)  - \alpha_k\nabla \overline{F}_{X_k}(w_k) \|]}_{\text{term 1}}  \nonumber \\
             &\quad +\frac{\alpha_k}{\gamma - \epsilon_H}\underbrace{\mathbb{E}_k[\|\nabla \overline{F}_{X_k}(w_k) - \nabla^2 \overline{F}_{S_k}(w_k)p_k\|]}_{\text{term 2}}.
\end{align}
Term 1 can be bounded as
\begin{align}
  &\frac{1}{\gamma - \epsilon_H}\mathbb{E}_k[\| \nabla^2 \overline{F}_{S_k}(w_k)(w_k -w^*)  - \alpha_k\nabla \overline{F}_{X_k}(w_k) \|] \nonumber \\ 
  &\leq \frac{1}{\gamma - \epsilon_H}\bigg[\frac{M}{2}\|w_k - w^*\|^2 + \bigg(L_{N_{S_k}}|1 - \alpha_k| +\frac{\sigma}{\sqrt{N_{S_k}}}\bigg)\|w_k - w^*\| + \frac{\alpha_k v}{\sqrt{N_{X_k}}}\bigg]. 
\end{align}
By  Lemma \ref{stochastic_newton_bound_alpha_k}. For term 2 we have that by assumption
\begin{equation}
  \mathbb{E}_k[\|\nabla \overline{F}_{X_k}(w_k) - \nabla^2 \overline{F}_{S_k}(w_k)p_k\|] \leq  \mathbb{E}_k [\eta_k\|\nabla \overline{F}_{X_k}(w_k)\|] \leq \mathbb{E}_k[ \|\nabla \overline{F}_{X_k}(w_k)\|^2].
\end{equation}
By a bound given in Theorem 2.1 in \cite{BollapragadaByrdNocedal2018} we have the Monte Carlo approximation error:
\begin{equation}
  \mathbb{E}_k[ \|\nabla \overline{F}_{X_k}(w_k) - \nabla \overline{F}(w_k)\|^2] \leq \frac{v^2}{N_{X_k}}.
\end{equation}
Employing the reverse triangle inequality we have
\begin{equation} \label{reverse_treq_bound}
  \mathbb{E}_k[ \|\nabla \overline{F}_{X_k}(w_k)\|^2] \leq \bigg(\|\nabla \overline{F}(w_k)\| + \frac{v}{\sqrt{N_{X_k}}}\bigg)^2.
\end{equation}
By Lemma 1.2 in \cite{EisenstatWalker1996} we have that
\begin{equation} \label{eisenbound}
   \|\nabla \overline{F}(w_k)\| \leq \mu \|w_k- w^*\|.
\end{equation}
So combining equations \eqref{reverse_treq_bound} and \eqref{eisenbound} we get
\begin{equation}
  \mathbb{E}_k[\|\nabla \overline{F}_{X_k}(w_k)\|^2] \leq \mu^2 \|w_k - w^*\|^2 +\frac{2v\mu}{\sqrt{N_{X_k}}}\|w_k - w^*\| + \frac{v^2}{N_{X_k}}.
\end{equation}
\end{proof}

\subsection{Proof of Theorem \ref{lowrankconvergence}}\label{low_rank_appendix}
\enskip \\

\begin{lemma}\label{newton_iterates_low_rank}
Let $\{w_k\}$ be the iterates generated by \eqref{trunc_newton_update}, and suppose that assumptions \ref{assumption_a1}-\ref{assumption_a3} hold, then for each $k$

\begin{align} \label{low_rank_iterates_eq}
  \mathbb{E}_k[\|w_{k+1} - w^*\|] \leq& \nonumber \\\frac{1}{|\overline{\lambda_{r}^{(k)}}+\gamma|}\bigg[&\frac{M}{2}\|w_k - w^*\|^2 +L_{N_{S_k}} |1 - \alpha_k|\|w_k - w^*\| \nonumber \\
    &+ \mathbb{E}_k\|([H_k^{(r)} + \gamma I](w_k) - \nabla^2 \overline{F} (w_k))(w_k - w^*)\| + \frac{\alpha_k v}{\sqrt{N_{X_k}}} \bigg],
\end{align}
where $\overline{\lambda_{r}^{(k)}} = \mathbb{E}_k[\lambda_{r}^{(k)}]$.
\end{lemma}

\begin{proof}
Expanding the left hand side of equation \eqref{low_rank_iterates_eq} and using the triangle inequality we can derive the following bound:
\begin{align}
  &\mathbb{E}_k[\|w_{k+1} - w^*\|] \nonumber\\
  & =\mathbb{E}_k[ \|w_k - w^* - \alpha_k[H_k^{(r)} + \gamma I]^{-1}(w_k)\nabla \overline{F}_{X_k}(w_k)\|] \nonumber\\
             &= \mathbb{E}_k[\|[H_k^{(r)} + \gamma I]^{-1}([H_k^{(r)} + \gamma I](w_k)(w_k - w^*) \nonumber \\
             & \quad - \alpha_k\nabla \overline{F}(w_k) - \alpha_k\nabla \overline{F}_{X_k}(w_k) + \alpha_k\nabla \overline{F}(w_k))\|] \nonumber \\
             &\leq \frac{1}{|\overline{\lambda_{r}^{(k)}}+\gamma|}\mathbb{E}_k[\|\big([H_k^{(r)} + \gamma I](w_k) - \nabla^2\overline{F}(w_k)\big) (w_k - w^*) \nonumber \\
             & \qquad\qquad\qquad+ \nabla^2 \overline{F}(w_k)(w_k - w^*) - \alpha_k\nabla \overline{F}(w_k)\|] \nonumber\\
             & \quad+\frac{\alpha_k}{|\overline{\lambda_{r}^{(k)}}+\gamma|}\mathbb{E}_k[\|\nabla \overline{F}_{X_k}(w_k) - \nabla \overline{F}(w_k)\|].
\end{align}
Therefore,
\begin{align}
  \mathbb{E}_k[\|w_{k+1} - w^*\|] &\leq \underbrace{\frac{1}{|\overline{\lambda_{r}^{(k)}}+\gamma|}\|\nabla^2 \overline{F}(w_k)(w_k-w^*) - \alpha_k\nabla \overline{F}(w_k)\|}_{\text{term 1}} \nonumber\\
  &+ \underbrace{\frac{1}{|\overline{\lambda_{r}^{(k)}}+\gamma|} \mathbb{E}_k[\|[H_k^{(r)} + \gamma I] - \nabla^2 \overline{F}(w_k))(w_k - w^*)]}_{\text{term 2}} \nonumber \\
  &+ \underbrace{\frac{\alpha_k}{|\overline{\lambda_{r}^{(k)}}+\gamma|}\mathbb{E}_k [\|\nabla \overline{F}_{X_k}(w_k) - \nabla \overline{F}(w_k)\|]}_{\text{term 3}}.
\end{align}
For term 1, a bound is given by Lemma \ref{lipschitz_bound_alpha_k}
\begin{align}
   &\frac{1}{|\overline{\lambda_{r}^{(k)}}+\gamma|}\|\nabla^2\overline{F}(w_k)(w_k - w^*) - \nabla \overline{F}(w_k)\| \nonumber \\
  &\leq \frac{M}{2|\overline{\lambda_{r}^{(k)}}+\gamma|}\|w_k - w^*\|^2+\frac{L_{N_{S_k}}|1-\alpha_k|}{|\overline{\lambda_{r}^{(k)}}+\gamma|}\|w_k -w^*\|.
\end{align}
Term 3 can be bounded by a result given in Lemma 2.2 of \cite{BollapragadaByrdNocedal2018}:
\begin{align}
\mathbb{E}_k [ \|\nabla \overline{F}(w_k) - \nabla \overline{F}_{X_k}\|]  \leq \frac{v}{\sqrt{N_{X_k}}}.
\end{align}
\end{proof}

\begin{lemma}\label{bounds_for_exact_truncated Hessian}
Bounds for exact truncated Hessian approximation. Suppose that Assumptions \ref{assumption_a1} and \ref{assumption_a4} hold, then

\begin{equation} \label{exact_truncated_hess_eq}
  \mathbb{E}_k\|[H_k^{(r)} + \gamma I](w_k) - \nabla^2 \overline{F} (w_k))(w_k - w^*)\| \leq \bigg[\overline{|\lambda_{r+1}^{(k)}|} + \gamma + \frac{\sigma}{\sqrt{N_{S_k}}} \bigg]\|w_k - w^*\|.
\end{equation}

\end{lemma}

\begin{proof}
Expanding the left hand side of equation \eqref{exact_truncated_hess_eq}, and employing the triangle inequality, the following bound can be derived:
\begin{align}
\mathbb{E}_k\|([H_k^{(r)}(w_k) + \gamma I] - \nabla^2 \overline{F} (w_k))(w_k - w^*)\| \leq \nonumber\\
 \underbrace{\mathbb{E}_k\|([H_k^{(r)}(w_k) + \gamma I] - \nabla^2 \overline{F}_{S_k} (w_k))(w_k - w^*)\|}_{\text{term 1}} \\
 + \underbrace{\mathbb{E}_k\|(\nabla^2 \overline{F} (w_k)- \nabla^2 \overline{F}_{S_k} (w_k))(w_k - w^*)\|}_{\text{term 2}} .
\end{align}
The first term is bounded as 
\begin{align}
 & \quad\mathbb{E}_k\|([H_k^{(r)}(w_k)- \nabla^2 \overline{F}_{S_k} (w_k)) + \gamma I] (w_k - w^*)\|  \nonumber \\
  &\leq \mathbb{E}_k\|(H_k^{(r)}(w_k)- \nabla^2 \overline{F}_{S_k} (w_k)) (w_k - w^*)\|+ \gamma\|w_k - w^*\|  \\
  & \leq\mathbb{E}_k\|H_k^{(r)}(w_k)- \nabla^2 \overline{F}_{S_k} (w_k)\|\| (w_k - w^*)\| + \gamma\|w_k - w^*\| \nonumber \\
  &= (\overline{|\lambda_{r+1}^{(k)}|} + \gamma) \|w_k - w^*\|.
\end{align}
The second term is bounded by
\begin{equation}
  \frac{\sigma}{\sqrt{N_{S_k}}}\|w_k - w^*\|,
\end{equation}
via Lemma 2.3 in \cite{BollapragadaByrdNocedal2018}.
\end{proof}

\begin{lemma}\label{randomized_lemma}
Bounds for randomized truncated Hessian approximation. Suppose that assumptions \ref{assumption_a1} and \ref{assumption_a4} hold, and $2 \leq r \leq \frac{d}{2}$, and $H_k^{(r)}$ is calculated via randomized SVD, with random matrices drawn from a Gaussian  probability measure $\rho$. Then:
\begin{align}\label{randomized_lemma_eq}
  &\mathbb{E}_k\big[\mathbb{E}_\rho\|[H_k^{(r)} + \gamma I](w_k) - \nabla^2 \overline{F} (w_k))(w_k - w^*)\|\big] \nonumber \\
  &\leq \bigg(\bigg(1+4\frac{\sqrt{d(r+p)}}{p-1}\bigg)\overline{|\lambda_{r+1}^{(k)}|} + \gamma + \frac{\sigma}{\sqrt{N_{S_k}}}\bigg)\|w_k - w^*\|
\end{align}

\end{lemma}

\begin{proof}
Expanding the left hand side of equation \eqref{randomized_lemma_eq} and employing the triangle inequality, the following bound can be established:
\begin{align}
\mathbb{E}_k\big[\mathbb{E}_\rho\|([H_k^{(r)}(w_k) + \gamma I] - \nabla^2 \overline{F} (w_k))(w_k - w^*)\|\big] \leq \nonumber\\
 \underbrace{\mathbb{E}_k\big[\mathbb{E}_\rho\|([H_k^{(r)}(w_k) + \gamma I] - \nabla^2 \overline{F}_{S_k} (w_k))(w_k - w^*)\|\big]}_{\text{term 1}} \\
 + \underbrace{\mathbb{E}_k\|(\nabla^2 \overline{F} (w_k)- \nabla^2 \overline{F}_{S_k} (w_k))(w_k - w^*)\|}_{\text{term 2}} 
\end{align}

The first term is bounded as 
\begin{align}
 & \quad\mathbb{E}_k\big[\mathbb{E}_\rho\|([H_k^{(r)}(w_k)- \nabla^2 \overline{F}_{S_k} (w_k)) + \gamma I] (w_k - w^*)\|\big]  \nonumber \\
  &\leq \mathbb{E}_k\big[\mathbb{E}_\rho\|(H_k^{(r)}(w_k)- \nabla^2 \overline{F}_{S_k} (w_k)) (w_k - w^*)\|\big]+ \gamma\|w_k - w^*\|  \\
  & \leq\mathbb{E}_k\big[\mathbb{E}_\rho\|H_k^{(r)}(w_k)- \nabla^2 \overline{F}_{S_k} (w_k)\|\| (w_k - w^*)\|\big] + \gamma\|w_k - w^*\| \nonumber \\
  &\leq \mathbb{E}_k \bigg(\bigg(1+4\frac{\sqrt{d(r+p)}}{p-1}\bigg)|\lambda_{r+1}^{(k)}| + \gamma\bigg) \|w_k - w^*\| \nonumber \\
  &= \bigg(\bigg(1+4\frac{\sqrt{d(r+p)}}{p-1}\bigg)\overline{|\lambda_{r+1}^{(k)}|} + \gamma\bigg) \|w_k - w^*\|,
\end{align}
where the second to last bound comes from equation 1.8 in \cite{HalkoMartinssonTropp2011}. The second term is bounded by
\begin{equation}
  \frac{\sigma}{\sqrt{N_{S_k}}}\|w_k - w^*\|,
\end{equation}
via Lemma 2.3 in \cite{BollapragadaByrdNocedal2018}.
\end{proof}

\lowrankconv*

\begin{proof}
This result follows immediately from Lemma \eqref{newton_iterates_low_rank}, Lemma \eqref{bounds_for_exact_truncated Hessian} and Lemma \eqref{randomized_lemma}.
\end{proof}

\subsection{Proof of Theorem \ref{incg_convergence}} \label{proof_of_incg_convergence}
\enskip \\

\incgconvergencerestat*

\begin{proof}
Expanding the left hand side of equation \eqref{incg_bound_eq} and employing the triangle inequality, the following bound can be established:
\begin{align}
&\mathbb{E}_k[\|w_{k+1} - w^*\|] = \mathbb{E}_k [\|w_k - w^* + \alpha_k p_k^r\|] \nonumber\\
& \leq \underbrace{\mathbb{E}_k[\|w_ k - w^* - \alpha_k \nabla^2 \overline{F}_{S_k}(w_k)\nabla \overline{F}_{X_k}(w_k)  \|]}_{\text{term 1}} + \alpha_k\underbrace{\mathbb{E}_k[\| p_k^r - \nabla^2\overline{F}_{S_k}^{-1}(w_k)\nabla \overline{F}_{X_k}   \|]}_{\text{term 2}}.
\end{align}
The first term can be bounded as 
\begin{align}
  &\mathbb{E}_k[\|w_ k - w^* - \alpha_k \nabla^2 \overline{F}_{S_k}(w_k)\nabla \overline{F}_{X_k}(w_k)  \|] \leq  \nonumber \\
  &\frac{1}{\gamma - \epsilon_H}\mathbb{E}_k[\| \nabla^2 \overline{F}_{S_k}(w_k)(w_k -w^*)  - \alpha_k\nabla \overline{F}_{X_k}(w_k) \|] \nonumber \leq\\ 
  &\frac{1}{\gamma - \epsilon_H} \bigg[\frac{M}{2}\|w_k - w^*\|^2 + \bigg(L_{N_{S_k}}|1 - \alpha_k| +\frac{\sigma}{\sqrt{N_{S_k}}}\bigg)\|w_k - w^*\| + \frac{\alpha_k v}{\sqrt{N_{X_k}}}\bigg],
\end{align}
by Lemma \ref{stochastic_newton_bound_alpha_k}. Term 2 can be bounded by the worst case convergence of the CG algorithm as in Lemma 3.1 in \cite{BollapragadaByrdNocedal2018}

\begin{equation}
  \mathbb{E}_k[\| p_k^r - \nabla^2\overline{F}_{S_k}^{-1}(w_k)\nabla \overline{F}_{X_k}   \|] \leq 2\alpha_k \frac{L_{N_{S_k}}}{\gamma - \epsilon_H}\sqrt{\kappa_{N_{S_k}}} \bigg(\frac{\sqrt{\kappa_{N_{S_k}}} - 1}{\sqrt{\kappa_{N_{S_k}}} + 1 }\bigg)^r \|w_k - w^*\|.
\end{equation}

\end{proof}

\subsection{Proof of Theorem \ref{gmres_convergence}} \label{proof_of_gmres_convergence}
\enskip \\

\gmresconvrestat*

\begin{proof}
Expanding the left hand side of equation \eqref{minres_gmres_bd_eq} and employing the triangle inequality, the following bound can be established:
\begin{align}
  &\mathbb{E}_k[\|w_{k+1} - w^*\|] \nonumber\\
  & =\mathbb{E}_k[ \|w_k - w^* - \alpha_k p_k^r\|] \nonumber\\
             &= \mathbb{E}_k[\|\nabla^2 \overline{F}_{S_k}(w_k)^{-1}(\nabla^2\overline{F}_{S_k}(w_k)(w_k - w^*) \nonumber \\
             & \quad  - \alpha_k\nabla^2\overline{F}_{S_k}(w_k)p_k^r +\alpha_k\nabla \overline{F}_{X_k}(w_k) - \alpha_k\nabla \overline{F}_{X_k}(w_k))\|] \nonumber \\
             &\leq \frac{1}{\gamma - \epsilon_H}\underbrace{\mathbb{E}_k[\| \nabla^2 \overline{F}_{S_k}(w_k)(w_k -w^*)  - \alpha_k\nabla \overline{F}_{X_k}(w_k) \|]}_{\text{term 1}} \nonumber \\
             & \quad +\frac{\alpha_k}{\gamma - \epsilon_H}\underbrace{\mathbb{E}_k[\|\nabla \overline{F}_{X_k}(w_k) - \nabla^2 \overline{F}_{S_k}(w_k)p_k^r\|]}_{\text{term 2}}
\end{align}
Term 1 can be bounded as
\begin{align}
  & \frac{1}{\gamma - \epsilon_H}\mathbb{E}_k[\| \nabla^2 \overline{F}_{S_k}(w_k)(w_k -w^*)  - \alpha_k\nabla \overline{F}_{X_k}(w_k) \|] \leq \nonumber \\
   &\frac{1}{\gamma - \epsilon_H}\bigg[\frac{M}{2}\|w_k - w^*\|^2 + \bigg(L_{N_{S_k}}|1-\alpha_k| +\frac{\sigma}{\sqrt{N_{S_k}}}\bigg)\|w_k - w^*\| + \frac{\alpha _k v}{\sqrt{N_{X_k}}}\bigg]. 
\end{align}
By Lemma \ref{stochastic_newton_bound_alpha_k}. For term 2 we can use a generic Krylov residual error bound
\begin{equation}
  \mathbb{E}_k[\|\nabla \overline{F}_{X_k}(w_k) - \nabla^2 \overline{F}_{S_k}(w_k)p_k^r\|] \leq  \mathcal{E}\mathbb{E}_k[\|\nabla \overline{F}_{X_k}\|].
\end{equation}
The last bound is given by the mean value theorem and Hessian spectral bound for $ N_{X_k}$ from assumption \ref{assumption_a1} as in Lemma \ref{lipschitz_bound_alpha_k}:
\begin{align}
  \mathbb{E}_k[\|\nabla \overline{F}_{X_k}(w_k)\|] &= \mathbb{E}_k[\|\nabla \overline{F}_{X_k}(w_k) - \nabla \overline{F}_{X_k}(w^*) + \nabla \overline{F}_{X_k}(w^*) \|] \nonumber \\
  &\leq \mathbb{E}_k[\|\nabla \overline{F}_{X_k}(w_k) - \nabla \overline{F}_{X_k}(w^*)\|] + \mathbb{E}_k[\|\nabla \overline{F}_{X_k}(w^*) \|] \nonumber \\
  &\leq L_{N_{X_k}} \|w_k  - w^*\| + \epsilon_g.
\end{align}
For GMRES, we have due to Proposition 6.33 in \cite{Saad2003} 
\begin{equation}
  \mathcal{E} = \frac{L_{N_{S_k}}}{\gamma - \epsilon_H} \frac{C_m(\frac{a}{d})}{|C_m(\frac{c}{d})|}.
\end{equation}
For MINRES, we have due to Theorem 5.10 in \cite{Saad2003}
\begin{equation}
  \mathcal{E} = \bigg(1 - \frac{(\gamma - \epsilon_H)^2}{L_{N_{S_k}}^2}\bigg)^\frac{k}{2}.
\end{equation}

\end{proof}


\end{document}